\documentclass[a4paper]{amsart}
\usepackage{pifont}
\usepackage{graphicx}
\usepackage{amssymb}
\usepackage{float}
\usepackage{amsmath}
\usepackage{amsthm,amsfonts,bbm}
\usepackage{amscd}
\usepackage{mathrsfs}
\usepackage[all,2cell]{xy}
\usepackage{pict2e,epic,eepic,pstricks}
\usepackage{verbatim}

\UseAllTwocells \SilentMatrices
\newtheorem{thm}{Theorem}[section]

\newtheorem{cor}[thm]{Corollary}
\newtheorem{lem}[thm]{Lemma}

\newtheorem{prop}[thm]{Proposition}
\newtheorem{prop-def}[thm]{Proposition-Definition}
\theoremstyle{definition}
\newtheorem{Def}[thm]{Definition}
\theoremstyle{remark}

\newtheorem{exm}[thm]{\bf Example}
\numberwithin{equation}{section}

\def\la{\lambda}
\def\al{\alpha}

\def\deg{{\rm deg}}
\def\al{\alpha}
\def\e{\varepsilon}
\def\N{\mathbb{N}}

\def\Id{{\rm Id}}

\def\S{\mathcal{S}}
\def\Sum{\sum\limits}
\def\Prod{\prod\limits}
\newcommand{\lr}[1]{\langle\,#1\,\rangle}

\def\S{\mathcal{S}}

\def\Res{{\rm Res}}
\def\e{{\bf e}}
\def\x{{\bf x}}

\def\C{\mathbb{C}}
\def\ot{\otimes}
\def\V{\mathcal{V}}
\def\E{\mathcal{E}}
\def\G{\mathcal{G}}
\def\B{\mathcal{B}}

\def\c{\mbox{\textbf{c}}}
\def \SV {\mathcal{S}}

\begin{document}
\title[Characteristic Polynomials of Hyperstars and Hyperpaths]
{A Combinatorial Method for Computing Characteristic Polynomials of Starlike Hypergraphs}

\author[Y.-H. Bao, Y.-Z. Fan, Y. Wang and M. Zhu] {Yan-Hong Bao, Yi-Zheng Fan*, Yi Wang and Ming Zhu}

\date{\today}
\thanks{$^*$The corresponding author}


\maketitle

\dedicatory{}%
\commby{}%

\begin{abstract}
By using the Poisson formula for resultants and the variants of chip-firing game on graphs, we provide a combinatorial method for computing a class of
of resultants, i.e. the characteristic polynomials of the adjacency tensors of starlike hypergraphs including hyperpaths and hyperstars,
which are given recursively and explicitly.
\end{abstract}

\section{Introduction}
Here a tensor (or hypermatrix) refers to a multi-array of entries in some field,
   which can be viewed to be the coordinates of the classical tensor (as multilinear function) under an orthonormal basis.
The eigenvalues of a tensor were introduced by Qi \cite{Qi, Qi2} and Lim \cite{Lim} independently.
To find the eigenvalues of a tensor, Qi \cite{Qi, Qi2} introduced the characteristic polynomial of a tensor,
 which is defined to be a resultant of a system of homogeneous polynomials.
In general, there is no an explicit polynomial formula yet for resultants except some very special cases;
and many fundamental questions about resultants still remain open.

As we know, there are mainly three tools to compute a concrete resultant.
The first one is Koszul complex, whose terms is given by the graded tensor product of a polynomial algebra and an exterior algebra,
and the differential is built from objective polynomials in the resultant.
The resultant is exactly equal to a certain characteristic of the related Koszul complex.
The second one is generalized trace, which is defined by  Morozov and Shakirov \cite{MS2}.
Using the generalized traces and the Schur function, Hu et.al gave an expression of the characteristic polynomial of a tensor \cite{HHLQ}.
Shao, Qi and Hu gave a graph theoretic formula for the generalized trace \cite{SQH}.
As an application, Cooper and Dulte computed the characteristic polynomial of the adjacency tensor of a single edge hypergraph \cite{CD}.
The third tool is Poisson formula, which may provide an inductively computing method, see \cite[Chapter 13, Theorem 1.2]{GKZ} or \cite[Proposition 2.7]{Jou}.
For example, Cooper and Dutle computed the spectrum of the ``all ones" tensors using the Poisson formula, see \cite[Theorem 3]{CD2}.
We refer to \cite[Chapter 13]{GKZ} and \cite[Chapter 3]{CLO} for an overview of calculation of resultants.

Recently, spectral hypergraph theory is proposed to explore connections between
the structure of a uniform hypergraph and the eigenvalues of some related symmetric tensors.
Cooper and Dutle \cite{CD} proposed the concept of adjacency tensor for a uniform hypergraph.
Shao et.al \cite{SSW} proved that the adjacency tensor of a connected $k$-uniform hypergraph $G$ has a symmetric H-spectrum if and only if $k$ is even and $G$ is odd-bipartite.
This result gives a certification to check whether a connected even-uniform hypergraph is odd-bipartite or not.

The characteristic polynomial of a hypergraph is defined to be the characteristic polynomial of its adjacency tensor.
In this paper, we mainly aims to give a lower dimension formula to compute the characteristic polynomial of hypergraphs based on Poisson formula and variants of chip-firing game,
and give the characteristic polynomials of starlike hypergraphs including hyperstars and hyperpaths, recursively and explicitly.

For simplicity of notation, we denote $[n]=\{1, 2, \cdots, n\}$ and $[m, n]=\{m, m+1, \cdots, n\}$ for integers $m<n$.

\section{Preliminaries}

In this section, we mainly recall some basic notions and useful results on resultants and hypergraphs.

\subsection{Resultants} Let $F_1(x_1, \cdots, x_n), \cdots, F_n(x_1, \cdots, x_n)$ be $n$ homogeneous polynomials over $\C$ in variables $x_1, \cdots, x_n$,
and the degree of $F_i$ is $d_i>0$ for $i \in [n]$.
An important question is whether the system of equations
\begin{align}\label{general eqs}
\begin{cases}
F_1(x_1, \cdots, x_n)=0,\\
~~~\cdots \\
F_n(x_1, \cdots, x_n)=0
\end{cases}
\end{align}
admits nontrivial solutions.

Generally, each $F_i$ can be written as
\[F_i=\sum_{|\al|=d_i} c_{i,\al}\x^\al,\]
where $\al=(i_1, \cdots, i_n)$, $|\al|=i_1+\cdots+i_n$ and $\x^\al=x_1^{i_1}\cdots x_n^{i_n}$. Note that the number of $\al$'s with $|\al|=d$ is $n+d-1 \choose d-1$.

For each possible pair of indices $i, \al$, we introduce a variable $u_{i, \al}$.
Then, given a polynomial $P\in \C[u_{i, \al}\colon |\al|=d_i, i\in [n]]$,
we let $P(F_1, \cdots, F_n)$ denote the value obtained by replacing each variable $u_{i, \al}$ in $P$ with the corresponding coefficient $c_{i, \al}$.

\begin{thm}\cite[Chapter 3, Theorem 2.3]{CLO} For fixed positive degrees $d_1, \cdots, d_n$, there exists a unique polynomial $\Res\in \mathbb{Z}[u_{i,\al}]$ satisfying the following properties:

\begin{enumerate}
\item[(i)] If $F_1, \cdots, F_n\in \C[x_1, \cdots, x_n]$ are homogeneous of degrees $d_1, \cdots, d_n$ respectively, the system \eqref{general eqs} has
a nontrivial solution if and only if $\Res(F_1, \cdots, F_n)=0$.

\item[(ii)] $\Res(x_1^{d_1}, \cdots, x_n^{d_n})=1$.

\item[(iii)] $\Res$ is irreducible, even regarded as a polynomial in $\C[u_{i,\al}]$.
\end{enumerate}
\end{thm}

$\Res(F_1, \cdots, F_n)$ is called the \emph{resultant} of $F_1, \cdots, F_n$.
Resultants have an important application in algebraic geometry, algebraic combinatorics and spectral hypergraph theory.
However, it is very difficult to compute the resultant of general polynomials.
Here, we only list some useful properties and calculation methods of resultants which will be used in this paper.

\begin{lem}\cite[Lemma 3.2]{CD}\label{Res(FG)=ResFResG}
Let $F_1, \cdots, F_n\in \C[x_1, \cdots, x_n]$ be homogeneous polynomials of degree $d_1, \cdots, d_n$ respectively,
and let $G_1, \cdots, G_m\in \C[y_1, \cdots, y_m]$ be homogeneous polynomials of degree $\delta_1, \cdots, \delta_m$ respectively.
Then
\[\Res(F_1, \cdots, F_n, G_1, \cdots, G_m)=\Res(F_1, \cdots, F_n)^{\prod\limits_{j=1}^m\delta_j} \Res(G_1, \cdots, G_m)^{\prod\limits_{i=1}^nd_i}.\]
\end{lem}


\begin{lem}\cite[Chapter 3, Theorem 3.1]{CLO} \label{Res(la)=laRes}
For a fixed $j \in [n]$,
\[\Res(F_1, \cdots, \la F_j, \cdots, F_n)=\la^{d_1\cdots d_{j-1}d_{j+1} \cdots d_n} \Res(F_1, \cdots, F_n),\]
where $d_i$ is the degree of $F_i$ for each $i \in [n]$.
\end{lem}

Next, we recall the Poisson formula.
Given homogeneous polynomials $F_1, \cdots, F_n\in \mathbb{C}[x_1, \cdots, x_n]$ of degree $d_1, \cdots, d_n$ respectively, let
\begin{align}
f_i(x_1, \cdots, x_{n-1})&=F_i(x_1, \cdots, x_{n-1}, 1), \ \  (1\le i\le n)\\
{\bar F}_{i}(x_1, \cdots, x_{n-1})&=F_i(x_1, \cdots, x_{n-1}, 0), \ \ (1\le i\le n-1).
\end{align}
Observe that $\bar{F}_1, \cdots, \bar{F}_{n-1}$ are still
homogeneous in $\mathbb{C}[x_1, \cdots, x_{n-1}]$ of degree
$d_1, \cdots, d_{n-1}$ respectively, but $f_1,\ldots,f_n$ are not homogeneous in general.

\begin{lem}[Poisson formula]\label{lower number of eqs} Keep the above notation.
If $\Res(\bar{F}_1, \cdots, \bar{F}_{n-1})\neq 0$, then the quotient algebra
$A=\dfrac{\mathbb{C}[x_1, \cdots, x_{n-1}]}{\lr{f_1, \cdots, f_{n-1}}}$ has dimension $d_1\cdots d_{n-1}$ as a vector space over $\mathbb{C}$,
where $\lr{f_1, \cdots, f_{n-1}}$ is the ideal of the polynomial algebra $\mathbb{C}[x_1, \cdots, x_{n-1}]$ generated
by $f_1, \cdots, f_{n-1}$, and
\begin{align}
\Res(F_1, \cdots, F_{n})=\Res({\bar F}_1, \cdots, {\bar F}_{n-1})^{d_n} \det(m_{f_{n}}\colon A \to A)
\end{align}
where $m_{f_n}\colon A \to A$ is the multiplication map given by $f_n$.
\end{lem}
Here, the above form of Poisson formula follows from \cite[Chapter 3, Theorem 3.4]{CLO},
which is different from the original one in \cite{Jou}.

\subsection{Hypergraphs}

A \emph{hypergraph} $H$ is a pair $(V, E)$, where $V$ is the set of vertices, and $E\subset \mathcal{P}(V)$ is the set of edges.
A hypergraph $H$ is called $k$-\emph{uniform} for an integer $k\ge 2$
if for each $\e\in E$, $|\e|=k$.
Clearly, a $2$-uniform hypergraph is just a classical simple graph.

\begin{Def}\cite{CD}
Let $H=(V, E)$ be a $k$-uniform hypergraph.
The (\emph{normalized}) \emph{adjacency tensor} $\mathcal{A}(H)=(a_{i_1\cdots i_k})_{i_1, \cdots, i_k\in V}$ is defined by
\begin{align*}
a_{i_1\cdots i_k}=\begin{cases}
\dfrac{1}{(k-1)!}, & \mbox{if~} \{i_1,\cdots, i_k\}\in E, \\
0, & {\rm otherwise}.
\end{cases}
\end{align*}
\end{Def}

For convenience, we use the following notation.
Let $V$ be a finite set and $m$ a positive integer.
For each $\textbf{e}=(i_1, \cdots, i_m)\in V^{m}$ and $\textbf{c}=(c_1, \cdots, c_m)\in \mathbb{N}^m$,  we denote
$\x_{\textbf{e}}^{\textbf{c}}=x_{i_1}^{c_1}\cdots x_{i_m}^{c_m}$. We also write $\x_{\textbf{e}}^{\mathbbm{1}}$ as $\x_\textbf{e}$, where
$\mathbbm{1}=(1, \cdots, 1)\in \mathbb{N}^m$. If $V=[n]$, $\textbf{c}=(c_1, \cdots, c_n) \in \mathbb{N}^n$, we write
$\x_{[n]}^{\textbf{c}}$ as $\x^{\textbf{c}}$.

The eigenvalues of a tensor was introduced by Qi \cite{Qi, Qi2} and Lim \cite{Lim} independently.
The adjacency tensor of a uniform hypergraph was introduced by Cooper and Dutle \cite{CD}.
Here we briefly give the definition of eigenvalues of uniform hypergraphs based on the above.

\begin{Def}\cite{Qi,CD}
Let $H=(V, E)$ be a $k$-uniform hypergraph and $\mathcal{A}=(a_{i_1\cdots i_k})$ be the adjacency tensor of $H$.
For some $\la\in \C$, if there exists a nonzero vector $\x\in \C^{|V|}$ such that for each $j \in V$,
\[\Sum_{i_2,i_3,\ldots,i_k \in V}a_{ji_2 i_3 \ldots i_k }x_{i_2}x_{i_3} \cdots x_{i_k}=\la x_j^{k-1},\]
or equivalently, for each $v\in V$,
\[\Sum_{v\in \e\in E} \x_{\e\backslash \{v\}}=\la x_v^{k-1},\]
then $\la$ is called an \emph{eigenvalue} of $H$.
\end{Def}

For each $v\in V$, define
\[F_v=\la x_v^{k-1}-\Sum_{v\in \e\in E} \x_{\e\backslash \{v\}}.\]
The polynomial
\[\phi_H(\la)=\Res(F_v\colon v\in V)\]
in the indeterminant $\la$ is called the \emph{characteristic polynomial} of $H$.
Consequently, $\la$ is an eigenvalue of $H$ if and only if $\phi_H(\la)=0$.

\subsection{Dollar game on graph}\label{Dollar game model}

Let $G=(V,E)$ be a simple graph.
Recall that a \emph{configuration} $\c$ on $G$ means a function $\c\colon V \to \N$, which can be understood
there is a pile of $\c(v)$ tokens (chips, or dollars) at each vertex $v$.
A dollar game on $G$ starts from a configuration $\c$.
At each step of the game, a vertex $v$ is \emph{fired}, that is, dollars move from $v$ to its adjacent vertices, one dollar going along each edge incident to $v$.
 Fix a vertex $w$ of $G$, called the \emph{bank vertex}.
 A vertex $v$ other than $w$ can be fired if and only if $\c(v)\ge \deg(v)$, where $\deg(v)$ is the degree of the vertex $v$.
The bank vertex $w$ is allowed to go into debt such that $w$ can be fired if and only if no other firing is possible.

Suppose that $\mathcal{X}$ is a non-empty finite sequence of (not necessarily distinct) vertices of $G$,
such that starting from a configuration $\c$, the vertices can be fired in the order of $\mathcal{X}$.
If $v$ occurs $x(v)$ times, we shall refer to $x$ as the representative vector for $\mathcal{X}$.
The configuration $\c'$ after the sequence of firing $\mathcal{X}$ is given by
\[\c'=\c-Lx\]
where $L$ is the Laplacian matrix of $G$.

The dollar game on graph was introduced by Biggs \cite{B},
  which is a variant of chip-firing game, and is often described in terms of ``snowfall" and ``avalanches" in the literature.
A configuration $\c$ is said to be \emph{stable} if $0\le \c(v)<\deg(v)$ for any $v\neq w$.
A sequence of firing is $w$-\emph{legal} if and only if each occurrence of a vertex $v\neq w$ follows a configuration
$t$ with $t(v)\ge \deg(v)$ and each occurrence of $w$ follows a stable configuration.
A configuration $\c$ on $G$ is said \emph{recurrent} if there is a $w$-legal sequence for $\c$ which leads to the same configuration.
A \emph{critical} configuration $\c$ means that $\c$ is both stable and recurrent. we refer to \cite{B} for more details.

\begin{lem}\cite[Theorem 6.2]{B} \label{critical conf}
If $G$ is a connected graph, then the number of critical configurations is equal to the number of
spanning trees of $G$.
\end{lem}

\begin{exm}\label{critical conf of completed graph}
Let $K_k$ be a completed graph on $k$ vertices. Then the number of critical configurations is $k^{k-2}$.
\end{exm}

\section{Poisson formula for characteristic polynomials of hypergraphs}
\def\c{\textbf{c}}
\subsection{Poisson formula for hypergraphs}

Let $H=(V, E)$ be a $k$-uniform hypergraph.
Recall that the characteristic polynomial of $H$ is defined as
\[\phi_H(\la)=\Res(F_v\colon v\in V),\]
where
\[F_v=\la x_v^{k-1}-\Sum_{v\in \e\in E} \x_{\e\backslash \{v\}}\in \C[x_v\colon v\in V].\]
In order to use Poisson formula for the resultant $\Res(F_v\colon v\in V)$, we need fix a vertex $w$ in $V$.
Denote by $E_w$ the set of all edges containing the vertex $w$ and $\e_{\widehat{w}}=\e\backslash \{w\}$ for each $\e\in E_w$.
Then we have
\begin{align}\label{F_v, v neq w}
f_w=&\la-\Sum_{\e\in E_w} \x_{\e_{\widehat{w}}}, \notag \\
f_v=& \la x_v^{k-1}-\Sum_{v\in \e\in E\backslash E_w}\x_{\e\backslash\{v\}}-\Sum_{v\in \e \in E_w} \x_{\e_{\widehat{w}}\backslash\{v\}}, \; v\in V\backslash \{w\},\notag  \\
\bar F_v=& \la x_v^{k-1}-\Sum_{v\in \e\in E\backslash E_w}\x_{\e\backslash\{v\}}, \; v\in V\backslash \{w\}.
\end{align}

Deleting the vertex $w$ in $V$ and the edges in $E_w$,
one can obtain a sub-hypergraph $\widehat{H}=(\widehat{V}, \widehat{E})$.
To be precise, $\widehat{V}=V\backslash \{w\}$ and $\widehat{E}=E\backslash E_w$.

\begin{lem}\label{lower dim formula}
 Retain the above notation. Then
\begin{align}\label{lower order}
\phi_H(\la)=\phi_{\widehat{H}}(\la)^{k-1}\det(m_{f_w}\colon A \to A),
\end{align}
where
$A$ is the quotient algebra $\dfrac{\C[x_v\colon v\in \widehat{V}]}{\lr{f_v\colon v\in \widehat{V}}}$
and $m_{f_w}$ is the multiplication map of $A$ given by $f_w$.
\end{lem}

\begin{proof}
By Lemma \ref{lower number of eqs}, the characteristic polynomial of $H$ is
\[\phi_H(\la)=\Res(\bar F_v\colon v\in \widehat{V})^{k-1}\det(m_{f_w}\colon A \to A).\]
Considering the subhypergraph $\widehat{H}$ of $H$, by Equation \eqref{F_v, v neq w},
we have
\[\phi_{\widehat{H}}(\la)=\Res(\bar F_v \colon v\in \widehat{V}).\]
The result follows.
\end{proof}

By definition, the algebra $A$ is $(k-1)^{r-1}$-dimensional as a vector space over $\C$
where $r$ is the number of vertices of $H$.
In general, it is difficult to compute the determinant $\det(m_{f_w}\colon A \to A)$.
However, we can give some description for some special cases.

\subsection{Hypergraphs with a cut vertex}

Let $H=(V, E)$ be a $k$-uniform connected hypergraph and $w\in V$.
Denote by $E_{\widehat{w}}=\{\e_{\widehat{w}}\mid \e\in E_w\}$.
Deleting the vertex $w$, we can get a (non-uniform) hypergraph $\widetilde{H}=(\widetilde{V}, \widetilde{E})$,
with $\widetilde{V}=\widehat{V}=V\backslash \{w\}$ and $\widetilde{E}=(E\backslash E_w)\cup E_{\widehat{w}}$.
Recall the vertex $w$ is called a \emph{cut vertex} if $\widetilde{H}$ is not connected; see Fig. \ref{cutvertex}.
Suppose that $w$ is a cut vertex and $\widetilde{H}_1=(\widetilde{V}_1, \widetilde{E}_1), \cdots, \widetilde{H}_n=(\widetilde{V}_n, \widetilde{E}_n)$ ($n\ge 2$) are
the connected components of $\widetilde{H}$.
For each $i\in [n]$, we set $V_i=\widetilde{V}_i$, $E_i=\widetilde{E}_i\backslash E_{\widehat{w}}$,
and then obtain a subhypergraph $H_i=(V_i, E_i)$ of $H$.
Note that each $H_i$ is a $k$-uniform hypergraph and may not be connected.

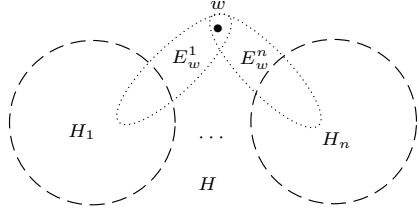
\begin{figure}[H]
\centering
\ifpdf
  \setlength{\unitlength}{1bp}%
  \begin{picture}(163.13, 90.66)(0,0)
  \put(0,0){\includegraphics{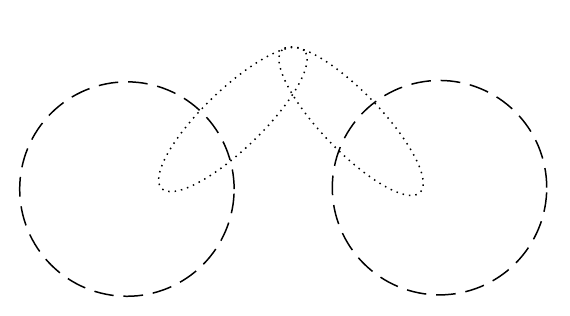}}
  \put(81.43,70.20){\fontsize{7.11}{8.54}\selectfont $\bullet$}
  \put(27.53,31.31){\fontsize{7.11}{8.54}\selectfont $H_1$}
  \put(122.03,28.16){\fontsize{7.11}{8.54}\selectfont $H_n$}
  \put(76.01,29.21){\fontsize{7.11}{8.54}\selectfont $\cdots$}
  \put(83.53,79.44){\fontsize{7.11}{8.54}\selectfont \makebox[0pt]{$w$}}
  \put(79.51,11.19){\fontsize{7.11}{8.54}\selectfont \makebox[0pt]{$H$}}
  \put(71.81,58.96){\fontsize{7.11}{8.54}\selectfont \makebox[0pt]{$E_w^1$}}
  \put(97.18,58.61){\fontsize{7.11}{8.54}\selectfont \makebox[0pt]{$E_w^n$}}
  \end{picture}
  \caption{A $k$-uniform hypergarph $H$ with a cut vertex $w$}\label{cutvertex}
\end{figure}

For each $i\in [n]$, we denote
\[E_w^i=\{\e\in E_w \mid \e\cap V_i\neq \varnothing \}.\]
By definition, we have
\[\phi_H(\la)=\Res(F_v\colon v\in V),\]
where
\begin{align*}
F_{v_i}=&\la x_{v_i}^{k-1}-\Sum_{v_i\in \e\in E_i} \x_{\e\backslash \{v_i\}}-\Sum_{v_i\in\e\in E_w^i} \x_{ \e\backslash \{v_i\}}, \; v_i\in V_i, i \in [n],\\
F_w=& \la x_w^{k-1}-\Sum_{i=1}^n\Sum_{\e\in E_w^i} \x_{\e_{\widehat{w}}},
\end{align*}
Therefore,
\begin{align*}
f_w=& \la -\Sum_{i=1}^n\Sum_{\e\in E_w^i} \x_{\e_{\widehat{w}}},\\
f_{v_i}=&\la x_{v_i}^{k-1}-\Sum_{v_i\in \e\in E_i} \x_{\e\backslash \{v_i\}}-\Sum_{v_i\in \e\in E_w^i} \x_{\e_{\widehat{w}} \backslash \{v_i\}}, \; v_i\in V_i,i \in [n],\\
\bar F_{v_i}=&\la x_{v_i}^{k-1}-\Sum_{v_i\in \e\in E_i} \x_{\e\backslash \{v_i\}}, \; v_i\in V_i, i \in [n].
\end{align*}

Let $A$ be the quotient algebra $\dfrac{\C[x_v\colon v\in \widehat{V}]}{\lr{f_v\colon v\in \widehat{V}}}$, and $m_{f_w}\colon A \to A$ is the linear map given by multiplication by $f_w$.
Since $V_1, \cdots, V_n$ form a partition of $\widehat{V}$ and for each $v_i\in V_i$, $f_{v_i}\in \C[x_v\colon v\in V_i]$,
we have $A=A_1 \otimes \cdots \otimes A_n$, where
\[A_i= \dfrac{\C[x_v\colon v\in V_i]}{\lr{f_v\colon v\in V_i}}, \; i \in [n].\]
We define and denote $m_{i, w}\colon A_i \to A_i$ the linear map given by the multiplication
  by $\Sum_{\e\in E_w^i} \x_{\e_{\widehat{w}}}$ for each $i\in [n]$.
Then
\[m_{f_w}=\la \Id_{A} -\Sum_{i=1}^n \Id_{A_1}\otimes \cdots \Id_{A_{i-1}}\otimes m_{i,w} \ot \Id_{A_{i+1}}\otimes \cdots \Id_{A_n},\]
where $\Id$ denotes the identity map on certain vector space.

\textbf{Assumption 1}.
\emph{For each $i\in [n]$, there exists an ordered $\C$-basis $\x^{\al_{i, 1}}, \cdots, \x^{\al_{i, d_i}}$ for $A_i$ such that
the matrix of $m_{i, w}$ with repect to this basis is a lower triangular matrix with the diagonal entry $\alpha_{i, j_i}$, $ j_i \in [d_i]$, where $d_i=(k-1)^{r_i}$ and $r_i=|V_i|$.}

Under the Assumption 1, $\{\x^{\al_{1, j_1}}\cdots \x^{\al_{n, j_n}}\mid j_i\in [d_i], i\in [n]\}$
with the left lexicographic order is a basis for $A$ such that the matrix of $m_{f_w}$ with respect to this basis
is still a lower triangular matrix with diagonal entries
$\la-\Sum_{1\le i\le n}\alpha_{i, j_i}$ for $j_i \in [d_i]$ and $i \in [n]$.
In this situation, we have
\begin{align}\label{det for cut vertex}
\det(m_{f_w}\colon A \to A)=\prod\limits_{1\le j_i\le d_i\atop 1\le i\le n} (\la-\Sum_{i=1}^n \alpha_{i, j_i})
\end{align}

\begin{cor}\label{charcut}
Let $H$ be a $k$-uniform hypergraph with a cut vertex $w$. Then, under the Assumption 1,
\begin{align*}
\phi_H(\la)=\prod\limits_{i=1}^n\phi_{H_i}(\la)^{(k-1)^{1+\Sum_{j\neq i}r_j}}
\prod\limits_{1\le j_i\le d_i\atop 1\le i\le n} (\la-\Sum_{i=1}^n \alpha_{i, j_i}).
\end{align*}
\end{cor}

\begin{proof} By definition, the characteristic polynomial of $H_i=(V_i, E_i)$ is
\[\phi_{H_i}(\la)=\Res(\bar{F}_{v_i} \colon v_i\in V_i).\]
By Lemma \ref{Res(FG)=ResFResG}, we have
\begin{align*}
\Res(\bar F_v\colon v\in \widehat{V})= \prod\limits_{i=1}^n \Res(\bar F_{v_i}\colon v_i\in V_i)^{(k-1)^{\Sum_{j\neq i} |V_j|}}
= \prod\limits_{i=1}^n \phi_{H_i}(\la)^{(k-1)^{\Sum_{j\neq i} r_j}}
\end{align*}
From Theorem \ref{lower dim formula}, it follows that
\begin{align*}
\phi_H(\la)=&\Res(\bar{F}_v\colon v\in \hat{V})^{k-1} \det(m_{f_w}\colon A \to A)\\
=&\prod\limits_{i=1}^n\phi_{H_i}(\la)^{(k-1)^{1+\Sum_{j\neq i}r_j}}\prod\limits_{j_i \in [d_i]\atop i \in [n]} (\la-\Sum_{i=1}^n \alpha_{i, j_i}).
\end{align*}
\end{proof}

\subsection{Hypergraphs with a cored vertex}\label{hg with a pendant vertex}

Let $H=(V, E)$ be a $k$-uniform hypergraph.
Recall that a vertex $w \in V$ is called a \emph{cored vertex} if it is contained in only one edge; see Fig. \ref{cored}.
Deleting the cored vertex $w$ and the edge $\e_w$ containing $w$,
one can obtain a sub-hypergraph $\widehat{H}=(\widehat{V}, \widehat{E})$ with $\widehat{V}=V\backslash \{w\}$
and $\widehat{E}=E\backslash \{\e_w\}$.
Then
\begin{align*}
  F_w & = \la x_w^{k-1}-\x_{\e_{\widehat{w}}}, \\
 F_v &=\la x_v^{k-1}-\Sum_{v\in \e\in \widehat{E}} \x_{\e\backslash \{v\}}-\Sum_{v\in \e_w} \x_{\e_w\backslash \{v\}}, \; v\neq w.
\end{align*}
Moreover,
\begin{align*}
   f_w & = \la -\x_{\e_{\widehat{w}}},\\
   f_v &=\la x_v^{k-1}-\Sum_{v\in \e\in \widehat{E}} \x_{\e\backslash \{v\}}-\Sum_{v\in \e_w} \x_{\e_{\widehat{w}}\backslash \{v\}},\; v\neq w,\\
 \bar F_v &= \la x_v^{k-1}-\Sum_{v\in \e\in \widehat{E}} \x_{\e\backslash \{v\}}, \; v\neq w.
\end{align*}

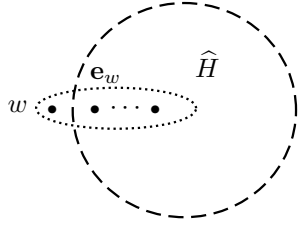
\begin{figure}[H]
\centering
\ifpdf
  \setlength{\unitlength}{1bp}%
  \begin{picture}(118.79, 91.84)(0,0)
  \put(0,0){\includegraphics{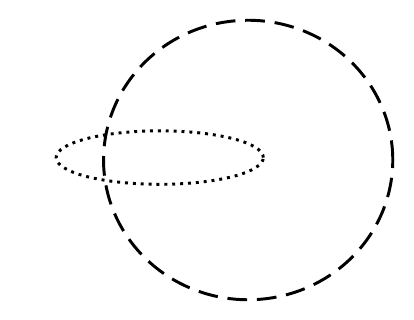}}
  \put(36.12,44.52){\fontsize{7.11}{8.54}\selectfont $\bullet$}
  \put(5.67,45.04){$w$}
  \put(43.64,44.52){$\cdots$}
  \put(58.59,44.52){\fontsize{7.11}{8.54}\selectfont $\bullet$}
  \put(20.19,44.52){\fontsize{7.11}{8.54}\selectfont $\bullet$}
  \put(75.49,58.52){$\widehat{H}$}
  \put(36.29,57.59){$\textbf{e}_w$}
  \end{picture}%
\caption{A hypergraph $H$ with a cored vertex $w$}\label{cored}
\end{figure}

\begin{cor}
Let $H$ be a $k$-hypergraph with a cored vertex $w$.
Retain the above notation.
Then
\[\phi_H(\la)=\phi_{\widehat{H}}(\la)^{k-1} \det(m_{f_w}\colon A\to A),\]
where $A$ is the quotient algebra $\dfrac{\C[x_v\colon v\in \widehat{V}]}{\lr{f_v\colon v\in \widehat{V}}}$ and $m_{f_w}$
is the multiplication map of $A$ given by $f_w$.
\end{cor}

\section{Hyperpaths}
\def\W{\omega}

Let $P_n^k$ be a $k$-uniform \emph{hyperpath} with $n$ edges or of length $n$, which has vertices labelled as $0, 1, \cdots, r=n(k-1)$ from left to right as in Fig. \ref{path},
and edges $\e_t=\{(k-1)(t-1), (k-1)(t-1)+1, \cdots, (k-1)t\}$ for $ t \in [n]$.

\begin{figure}[H]
\centering
\ifpdf
  \setlength{\unitlength}{1bp}%
  \begin{picture}(237.71, 45.55)(0,0)
  \put(0,0){\includegraphics{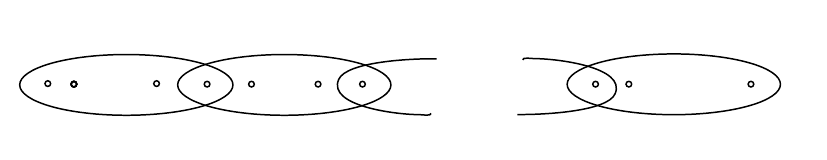}}
  \put(10.92,7.20){\fontsize{7.11}{8.54}\selectfont 0}
  \put(19.49,7.72){\fontsize{7.11}{8.54}\selectfont 1}
  \put(52.57,7.38){\fontsize{7.11}{8.54}\selectfont $k-1$}
  \put(96.49,7.38){\fontsize{7.11}{8.54}\selectfont $2(k-1)$}
  \put(203.07,7.55){\fontsize{7.11}{8.54}\selectfont $n(k-1)$}
  \put(26.84,19.62){\fontsize{7.11}{8.54}\selectfont $\cdots$}
  \put(76.54,19.80){\fontsize{7.11}{8.54}\selectfont $\cdots$}
  \put(189.59,19.45){\fontsize{7.11}{8.54}\selectfont $\cdots$}
  \put(76.54,19.80){\fontsize{7.11}{8.54}\selectfont $\cdots$}
  \put(35.42,34.32){\fontsize{7.11}{8.54}\selectfont $\textbf{e}_1$}
  \put(78.99,34.32){\fontsize{7.11}{8.54}\selectfont $\textbf{e}_2$}
  \put(189.42,34.32){\fontsize{7.11}{8.54}\selectfont $\textbf{e}_n$}
  \put(133.59,19.45){\fontsize{7.11}{8.54}\selectfont $\cdots$}
  \end{picture}%
\caption{A $k$-uniform hyperpath $P_n^k$ with $n$ edges}\label{path}
\end{figure}
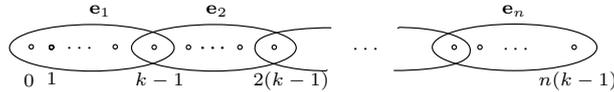

In this section, we will give a recursive formula for the characteristic polynomials of hyperpaths.
By the Poisson formula introduced in Theorem \ref{lower dim formula}, it suffices to compute the related determinant.
For this, we need introduce the following model of dollar game on hypergraphs.

\subsection{Dollar game on hypergraph and firing graph}
We now define a dollar game on a hypergraph, considered as a variant of dollar game on a graph.
Let $H=(V, E)$ be a $k$-uniform hypergraph with a specified bank vertex $w$.
A function $\c\colon V \to \mathbb{N}$ is called a \emph{configuration on} $H$.
A dollar game starts from a configuration $\c$.
At each of step of the game, a vertex $v$ is \emph{fired}, that is, the vertex $v$ decreases $k-1$ dollars, and each of its adjacent vertices increases $1$ dollar.
A vertex $v$ other than $w$ can be fired if and only if $\c(v)\ge k-1$, and the bank vertex $w$ is allowed to go into debt such that
$w$ can be fired if and only if no other firing is possible.
We say that a configuration $\c$ is \emph{stable} if $0 \le \c(v)<k-1$ for any $v \ne w$.
The notions of $w$-legal sequence of firing, recurrent or critical configurations are same as those defined in Section 2.3.
The above setting of dollar game on hypergraphs is different from that of dollar game on simple graphs in Section 2.3, but will be useful for our discussion.

Let ``$\le $" be a total ordering on the set $\widehat{V}=V\backslash\{w\}$.
Let $\c$ be a configuration on $H$.
The \emph{weight} of $\c$ is defined and denoted to be $\W(\c)=\sum_{v\in \widehat V} \c(v)$.
%
We define the \emph{left anti-lexicographical order} $\prec$ for all of configurations on $H$.
To be precise, for any configurations $\c$ and $\c'$, $\c\prec \c' $ if and only if
either $\W(\c)<\W(\c')$, or $\W(\c)=\W(\c')$, $\c(i)=\c'(i)$ for any $1\le i\le t-1$ and $\c(t)>\c'(t)$ for some $t$.

Suppose further that the bank vertex $w$ is also a cored vertex of $H$.
Based on the above discussion, we now define a directed graph, called a \emph{firing graph} $\G(\c_0)$ associated with a stable configuration $\c_0$ on the hypergraph $H$,
  which is closely related the dollar game on $H$ starting from $\c_0$.
Here, a vertex of $\G(\c_0)$ is a configuration and a directed edge will be called an \emph{arrow}.


Step 1. Initially, set $\V_0=\{\c_0\}$ and $\E_0=\varnothing$.

Step 2. Let $\bar\c_0$ be a configuration given by
\[\bar\c_0(v)=\begin{cases}
\c_0(v)+1, & \mbox{if~}v\in \e_w\backslash \{w\}, \\
\c_0(v), & \mbox{otherwise},
\end{cases}\]
where $\e_w$ is the unique edge of $H$ containing $w$.
Put $\V_1=\V_0\cup \{\bar\c_0\}$ and $\E_1=\{\c_0\xrightarrow{\al_{w, \e_w}}\bar \c_0\}$, where the arrow in $\E_1$ means
that the configuration $\bar\c_0$ is obtained from $\c_0$ by firing the bank vertex $w$ on the edge $\e_w$.

Step 3. If all configuration $\V_i\backslash \V_{i-1}$ are stable,
then we define $\V(\G(\c_0))=\V_i$ and $\E(\G(\c_0))=\E_i$.
Otherwise, for each non-stable configuration $\c\in \V_i\backslash \V_{i-1}$,
we choose the vertex $u_{\c}:=\max\{v\in \widehat{V} \mid \c(v)\ge k-1\}$. For each edge $\e\in E$ containing $u_{\c}$, we define
an arrow $\c \xrightarrow{\al_{u_{\c}, \e}} \bar\c_{\e}$, where the configuration $\bar\c_\e$ is given by
\[\bar\c_\e(v)=
\begin{cases}
\c(v)-(k-1), & \mbox{if~}v=u_{\c}, \\
\c(v)+1, & \mbox{if~} v\in \e\backslash \{u_{\c}\}, \\
\c(v), & \textrm{otherwise},
\end{cases}\]
and
\begin{align*}
\V_{i+1}&=\V_i\cup \{\bar\c_\e\mid u_{\c} \in \e\in E, \c\ \mbox{is not stable in}\ \V_i\backslash \V_{i-1}\},\\
\E_{i+1}&=\E_i\cup \{\c \xrightarrow{\al_{u_{\c}, \e}} \bar\c_{\e} \mid u_{\c}\in \e \in E, \c\ \mbox{is not stable in }\ \V_i\backslash \V_{i-1}\}.
\end{align*}
Note that $\bar\c_\e$ may have been in $\V_i$. If $\bar\c_\e \notin \V_i$, we say that $\bar\c_\e$ is obtained from $\c$ by firing $u_{\c}$ on $\e$.

The Step 3 tells us the \emph{firing rule}, that is, which vertex will be fired at next step among all non-stable vertices other than the bank vertex.
In addition, from the construction of $\G(\c_0)$, we have $0\le \W(\c) \le \W(\c_0)+k-1$ for any $\c\in \V(\G(\c_0))$,
which forces that $\G(\c_0)$ is a finite directed graph.

\begin{exm}\label{P33}
Let $P^3_3$ be the $3$-uniform hyperpath with $3$ edges as in Fig. \ref{path} by taking $n=3$ and $k=3$.
Let $0$ be the bank vertex of $P^3_3$, and let $\c_0=(*,1,1,1,1,0,0)$.
Here, the value of the bank vertex $0$ is omitted in each configuration of the dollar game starting from $\c_0$.
The firing graph $\G(\c_0)$ is drawn in Fig. \ref{P33}, where each vertex within a circle means that it will be fired at the next step.

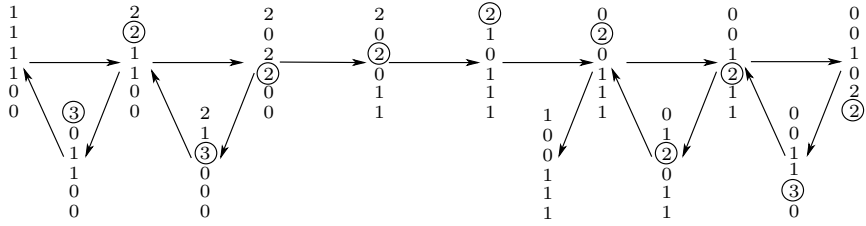
\begin{figure}[H]
\centering
\ifpdf
  \setlength{\unitlength}{1bp}%
  \begin{picture}(330.31, 93.99)(0,0)
  \put(0,0){\includegraphics{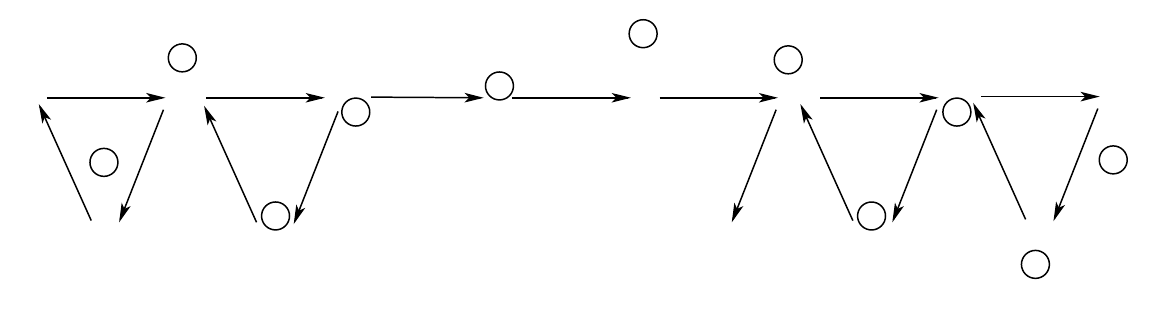}}
  \put(5.67,82.41){\fontsize{7.11}{8.54}\selectfont 1}
  \put(5.67,74.89){\fontsize{7.11}{8.54}\selectfont 1}
  \put(5.67,67.36){\fontsize{7.11}{8.54}\selectfont 1}
  \put(5.67,59.84){\fontsize{7.11}{8.54}\selectfont 1}
  \put(5.67,53.07){\fontsize{7.11}{8.54}\selectfont 0}
  \put(5.67,45.55){\fontsize{7.11}{8.54}\selectfont 0}
  \put(50.81,82.41){\fontsize{7.11}{8.54}\selectfont 2}
  \put(50.81,74.89){\fontsize{7.11}{8.54}\selectfont 2}
  \put(50.81,67.36){\fontsize{7.11}{8.54}\selectfont 1}
  \put(50.81,59.84){\fontsize{7.11}{8.54}\selectfont 1}
  \put(50.81,53.07){\fontsize{7.11}{8.54}\selectfont 0}
  \put(50.81,45.55){\fontsize{7.11}{8.54}\selectfont 0}
  \put(28.24,44.79){\fontsize{7.11}{8.54}\selectfont 3}
  \put(28.24,37.27){\fontsize{7.11}{8.54}\selectfont 0}
  \put(28.24,29.75){\fontsize{7.11}{8.54}\selectfont 1}
  \put(28.24,22.23){\fontsize{7.11}{8.54}\selectfont 1}
  \put(28.24,15.46){\fontsize{7.11}{8.54}\selectfont 0}
  \put(28.24,7.93){\fontsize{7.11}{8.54}\selectfont 0}
  \put(100.77,81.85){\fontsize{7.11}{8.54}\selectfont 2}
  \put(100.77,74.32){\fontsize{7.11}{8.54}\selectfont 0}
  \put(100.77,66.80){\fontsize{7.11}{8.54}\selectfont 2}
  \put(100.77,59.28){\fontsize{7.11}{8.54}\selectfont 2}
  \put(100.77,52.51){\fontsize{7.11}{8.54}\selectfont 0}
  \put(100.77,44.98){\fontsize{7.11}{8.54}\selectfont 0}
  \put(77.08,44.61){\fontsize{7.11}{8.54}\selectfont 2}
  \put(77.08,37.08){\fontsize{7.11}{8.54}\selectfont 1}
  \put(77.08,29.56){\fontsize{7.11}{8.54}\selectfont 3}
  \put(77.08,22.04){\fontsize{7.11}{8.54}\selectfont 0}
  \put(77.08,15.27){\fontsize{7.11}{8.54}\selectfont 0}
  \put(77.08,7.74){\fontsize{7.11}{8.54}\selectfont 0}
  \put(142.15,81.85){\fontsize{7.11}{8.54}\selectfont 2}
  \put(142.15,74.32){\fontsize{7.11}{8.54}\selectfont 0}
  \put(142.15,66.80){\fontsize{7.11}{8.54}\selectfont 2}
  \put(142.15,59.28){\fontsize{7.11}{8.54}\selectfont 0}
  \put(142.15,52.51){\fontsize{7.11}{8.54}\selectfont 1}
  \put(142.15,44.98){\fontsize{7.11}{8.54}\selectfont 1}
  \put(183.53,81.85){\fontsize{7.11}{8.54}\selectfont 2}
  \put(183.53,74.32){\fontsize{7.11}{8.54}\selectfont 1}
  \put(183.53,66.80){\fontsize{7.11}{8.54}\selectfont 0}
  \put(183.53,59.28){\fontsize{7.11}{8.54}\selectfont 1}
  \put(183.53,52.51){\fontsize{7.11}{8.54}\selectfont 1}
  \put(183.53,44.98){\fontsize{7.11}{8.54}\selectfont 1}
  \put(225.31,81.85){\fontsize{7.11}{8.54}\selectfont 0}
  \put(225.31,74.32){\fontsize{7.11}{8.54}\selectfont 2}
  \put(225.31,66.80){\fontsize{7.11}{8.54}\selectfont 0}
  \put(225.31,59.28){\fontsize{7.11}{8.54}\selectfont 1}
  \put(225.31,52.51){\fontsize{7.11}{8.54}\selectfont 1}
  \put(225.31,44.98){\fontsize{7.11}{8.54}\selectfont 1}
  \put(204.92,44.06){\fontsize{7.11}{8.54}\selectfont 1}
  \put(204.92,36.54){\fontsize{7.11}{8.54}\selectfont 0}
  \put(204.92,29.02){\fontsize{7.11}{8.54}\selectfont 0}
  \put(204.92,21.49){\fontsize{7.11}{8.54}\selectfont 1}
  \put(204.92,14.72){\fontsize{7.11}{8.54}\selectfont 1}
  \put(204.92,7.20){\fontsize{7.11}{8.54}\selectfont 1}
  \put(273.90,81.85){\fontsize{7.11}{8.54}\selectfont 0}
  \put(273.90,74.32){\fontsize{7.11}{8.54}\selectfont 0}
  \put(273.90,66.80){\fontsize{7.11}{8.54}\selectfont 1}
  \put(273.90,59.28){\fontsize{7.11}{8.54}\selectfont 2}
  \put(273.90,52.51){\fontsize{7.11}{8.54}\selectfont 1}
  \put(273.90,44.98){\fontsize{7.11}{8.54}\selectfont 1}
  \put(249.31,44.40){\fontsize{7.11}{8.54}\selectfont 0}
  \put(249.31,36.88){\fontsize{7.11}{8.54}\selectfont 1}
  \put(249.31,29.35){\fontsize{7.11}{8.54}\selectfont 2}
  \put(249.31,21.83){\fontsize{7.11}{8.54}\selectfont 0}
  \put(249.31,15.06){\fontsize{7.11}{8.54}\selectfont 1}
  \put(249.31,7.54){\fontsize{7.11}{8.54}\selectfont 1}
  \put(319.29,82.18){\fontsize{7.11}{8.54}\selectfont 0}
  \put(319.29,74.66){\fontsize{7.11}{8.54}\selectfont 0}
  \put(319.29,67.14){\fontsize{7.11}{8.54}\selectfont 1}
  \put(319.29,59.61){\fontsize{7.11}{8.54}\selectfont 0}
  \put(319.29,52.84){\fontsize{7.11}{8.54}\selectfont 2}
  \put(319.29,45.32){\fontsize{7.11}{8.54}\selectfont 2}
  \put(296.73,44.57){\fontsize{7.11}{8.54}\selectfont 0}
  \put(296.73,37.04){\fontsize{7.11}{8.54}\selectfont 0}
  \put(296.73,29.52){\fontsize{7.11}{8.54}\selectfont 1}
  \put(296.73,23.51){\fontsize{7.11}{8.54}\selectfont 1}
  \put(296.73,15.23){\fontsize{7.11}{8.54}\selectfont 3}
  \put(296.73,7.70){\fontsize{7.11}{8.54}\selectfont 0}
  \end{picture}%
\caption{The firing graph of $P_3^3$ associated with $\c_0$}\label{P33}
\end{figure}
\end{exm}

Generally speaking, it is difficult to obtain the firing graphs $\G(\c_0)$ for all stable configurations $\c_0$ on a general hypergraph.
However, for some special classes of hypergraphs, e.g. hyperpaths,
we can characterize the structure of $\G(\c_0)$.

\begin{exm}
Let $P_1^k$ be the $k$-uniform hyperpath on vertices $0, 1, \cdots, k-1$ with a single edge, and let $0$ be the bank vertex.
To determine the firing graph associated to a given configuration $\c_0$,
it is equivalent to consider the dollar game on the completed graph $K_k$ with the bank vertex $0$.
To be precise, a configuration $\c_0$ is a critical stable configuration if and only if the firing graph $\G(\c_0)$ is
is a directed cycle of length $k$,
and $\c_0$ is a non-critical stable configuration if and only if $\G(\c_0)$ is a directed path of length less than $k-1$.
\end{exm}

\subsection{Firing graphs of hyperpaths}

In this part, we characterize the structure of the firing graph of $P_n^k$ in Fig. \ref{path} associated to arbitrary fixed stable configuration, where the vertex $0$ is the bank vertex.
For a configuration $\c$ on $[0,n(k-1)]$, we denote $\c=(\c(0),\c^1, \cdots, \c^n)$, where $\c^i$ is the restriction of $\c$
on $\widehat{\e}_i:=\e_i \backslash \{(i-1)(k-1)\}$, i.e.
\[\c^i=(\c((i-1)(k-1)+1), \cdots, \c(i(k-1))), i \in [n].\]
Let $\tilde\c^i$ be the restriction of $\c$ on $\e_i$, i.e.
\[\tilde\c^i=(\c((i-1)(k-1)),\c((i-1)(k-1)+1), \cdots, \c(i(k-1))), i \in [n].\]
Then for each $i \in [n]$, $\tilde\c^i$ can be considered as a stable configurations on the completed graph $K_k$ with vertex set $\e_i$ and bank vertex $(i-1)(k-1)$.

Let $\mathcal{S}_k$ be the set of all stable configurations on the completed graph $K_k$,
and $\mathcal{C}_k$ be the set of all critical stable configurations on $K_k$.
Denote by $\B$ the set of all stable configurations on $P_n^k$, and for $s\in [0, n-1]$
\[\B_s=\{(\c^1, \cdots, \c^n)\mid  \tilde\c^i\in \mathcal{C}_k \hbox{~for~} i \in [s], \tilde\c^{s+1}\notin \mathcal{C}_k, \hbox{~and~} \tilde\c^i\in \mathcal{S}_k \hbox{~for~} i \in [s+2, n]\},\]
and $\B_n=\{(\c^1, \cdots, \c^n)\mid \tilde\c^i\in \mathcal{C}_k, i\in [n]\}$. Clearly, $\mathcal{B}$ is the disjoint union of $\B_0, \cdots, \B_n$.

\begin{lem}\label{t(k-1) fired iff}
Let $P_n^k$ be the $k$-uniform hyperpath as in Fig. \ref{path}.
Suppose that $\c_0$ is a stable configuration on $P_n^k$ with $0$ as the bank vertex,  and $\G(\c_0)$
is the firing graph of $P_n^k$ associated to $\c_0$.
Then
\begin{itemize}
\item[(i)] for any $t\in [n]$, the vertex
$t(k-1)$ is fired at $\c\in \V(\G(\c_0))$ if and only if $\c(t(k-1))=k-1$.
\item[(ii)] for any configuration $\c\in \V(\G(\c_0))$, $\W(\c)\ge \W(\c_0)$,
where the equality holds if and only if the vertices $1, \cdots, k-1$ have been fired on $\e_1$ exactly once along any directed path from $\c_0$ to $\c$.
\end{itemize}
\end{lem}

\begin{proof} (i)
By the firing rule defined in Section 4.1, it suffices to show that
if $\c(t(k-1))=k-1$, then $\c(v)<k-1$ for any $v>t(k-1)$.
According to the construction, for each $\c\in \mathcal{V}(\G(\c_0))$,
there exists a directed path $P(\c_0, \c)$ from $\c_0$ to $\c$ (without repeated vertices on the path).

Clearly, if each configuration $\c'$ in $P(\c_0, \c)$ except $\c$ satisfies $\c'(t(k-1))<k-1$,
then $\c(v)=\c_0(v)<k-1$ for any $v>t(k-1)$ since each vertex $v>t(k-1)$ is first fired only after $t(k-1)$ is fired.
Otherwise, $t(k-1)$ is fired at some configuration say $\c'_{-1}$ before $\c$ along the path $P(\c_0, \c)$.

Assume to the contrary that $\c(t(k-1))=k-1$ and $\c(u)\ge k-1$ for some $u>t(k-1)$.
Consider the directed path $P(\c_0, \c)$:
\[ \xymatrix{
\c_0\ar[r] & \cdots\ar[r] & \c'_{-1}\ar[r]^{\al_{t(k-1), \ast}} &\c'_0\ar[r]^{\al_{u_0, \ast}} &  \c'_1\ar[r]^{\al_{u_1, \ast}}
&\cdots\ar[r] & \c'_m\ar[r]^{\al_{u_m, \ast}} & \c}\]
Furthermore, we can assume that along the path $P(\c_0, \c)$,
$\c$ first appears with $\c(t(k-1)=k-1$ and $\c(u)\ge k-1$ for some $u>t(k-1)$, and $t(k-1)$ is last fired (at $\c'_{-1}$).
So, $\c'_i(t(k-1))<k-1$ for $i \in [0, m]$.
We also have $\c'_{-1}(t(k-1))=k-1$ and $\c'_{-1}(v)<k-1$ for any $v>t(k-1)$ by the firing rule.

If the arrow $\c'_{-1}\xrightarrow{\al_{t(k-1), \ast}} \c'_0$ satisfies $\al_{t(k-1), \ast}=\al_{t(k-1), \e_t}$, then $u_i<t(k-1)$ for $i=0, 1, \cdots, m$, as $t(k-1)$ is no longer fired so that any vertex $v > t(k-1)$  keep value invariant given by the configurations form $\c'_{-1}$ to $\c$.
It follows that $$\c(v)=\c'_i(v)=\c'_{-1}(v)<k-1$$ for any $v>t(k-1)$ and $i \in [0, m]$, a contradiction.

If the arrow $\c'_{-1}\xrightarrow{\al_{t(k-1), \ast}} \c'_0$ satisfies $\al_{t(k-1), \ast}=\al_{t(k-1), \e_{t+1}}$, then
\[\c'_0(v)=\begin{cases}
0, & \hbox{~if~} v=t(k-1)\\
\c'_{-1}(v)+1 , & \hbox{~if~} v \in [t(k-1)+1,(t+1)(k-1)], \\
\c'_{-1}(v), & \hbox{~if~} v\ge (t+1)(k-1)+1
\end{cases}\]
It follows that the number of configurations in $\{\c'_i\mid \al_{u_i, \ast}=\al_{u_i, \e_{t+1}}, i=0, 1, \cdots, m\}$
is $k-1$ as $\c(t(k-1))=k-1$.
Observe that for any $v \in [t(k-1)+1,(t+1)(k-1)-1]$,
\[\c'_i(v)\le \c'_0(v)+(k-2)=c'_{-1}(v)+(k-1)<2k-2.\]
It forces that each vertex in $[t(k-1)+1,(t+1)(k-1)-1]$ is fired exactly once on the edge $\e_{t+1}$ from $\c'_0$ and $\c$,
implying that the vertex $(t+1)(k-1)$ is also fired exactly once on $\e_{t+1}$.
Let $u_j$ be the vertex in $[t(k-1)+1,(t+1)(k-1)]$ last fired on $\e_{t+1}$.
Then $\c_{u_{j+1}}(t(k-1))=k-1$, and $\c_{u_{j+1}}(v) < k-1$ for any $v > t(k-1)$ by the assumption on $u_j$ and firing rule.
As $t(k-1)$ is last fired at $\c'_{-1}$, $\c_{u_{j+1}}=\c$, a contradiction.
%
%
%

(ii) By definition, we have $\W(\bar\c_0)=\W(\c_0)+k-1$. Observe that for any arrow $\c\xrightarrow{\al_{u, \e_t}} \c'$,
$\W(\c')=\W(\c)-1$ if $\e_t=\e_1$, and $\W(\c')=\W(\c)$ if $\e_t\neq \e_1$.
Therefore, we need only show that no vertex is fired on $\e_1$ more than one time in any directed path starting from $\c_0$.
Let
\[\xymatrix{
\c_0 \ar[r]^{\al_{0,\e_1}} & \bar\c_0=\c_1 \ar[r]^{\al_{u_1, \ast}} &\cdots\ar[r]& \c_\ell\ar[r]^{\al_{u_\ell, \e_1}} & \cdots\ar[r] & \c_t\ar[r]^{\al_{u_t, \e_1}} & \cdots
}\]
be arbitrary directed path in $\G(\c_0)$.
If there exists a vertex in $[k-1]$ be fired more than one time on $\e_1$, we can
assume that $u_t$ is the first occurrence of the vertex satisfying $\al_{u_t, \e_1}=\al_{u_\ell, \e_1}$ for some $\ell < t$.
If $1\le u_\ell = u_t\le k-2$, by our assumption, we have
\[\c_t(u_t)\le \bar\c_0(u_t)-(k-1)+(k-2)=\c_0(u_t)<k-1,\]
since there are at most $k-2$ vertices different from $u_t$ in $\e_1\backslash\{0\}$, a contradiction.
If $u_\ell=u_t=k-1$, we have $\c_\ell(v)<k-1$ for any $v>k-1$ by the firing rule,
$u_i \in [k-2]$ for each $i \in [\ell+1,t-1]$ and each of them is fired only one time by our assumption on $u_t$.
By the part (i), we have $\c_\ell(k-1)=k-1$. Therefore,
$$\c_t(k-1)\le \c_\ell(k-1)-(k-1)+(k-2)=k-2.$$
This is also contradiction.
\end{proof}

Next, we describe the structure of the firing graph of $P_n^k$ associated to any stable configuration.

\begin{prop}\label{Linear Path Firing Graph}
Let $P_n^k$ be the $k$-uniform hyperpath as in Fig. \ref{path}.
Then, for a given stable configuration $\c_0\in \mathcal{B}_s$ on $P_n^k$ with $0$ as the bank vertex, where $s\in [n]$,
the firing graph $\G(\c_0)$ of $P_n^k$ associated to $\c_0$ has the structure as in Fig. \ref{firegraph},
where $\G'$ is the subgraph induced by the configurations obtained by first firing the vertex $s(k-1)$ on the edge $\e_{s+1}$,
each directed cycle in $\G(\c_0)\backslash \G'$ has length $k$,
such that
\begin{enumerate}
\item [(i)]$\G'$ is not empty if and only if $s<n$;

\item[(i)] if $s<n$, there does not exist arrows from $\G'$ to $\G(\c_0)\backslash \G'$;

\item[(ii)] if $s<n$, $\c_0\prec \c'$ for any $\c'\in \G'$.
\end{enumerate}
\end{prop}

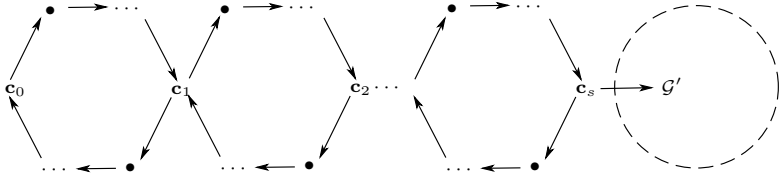
\begin{figure}[H]
\centering
\ifpdf
  \setlength{\unitlength}{1bp}%
  \begin{picture}(300.97, 79.76)(0,0)
  \put(0,0){\includegraphics{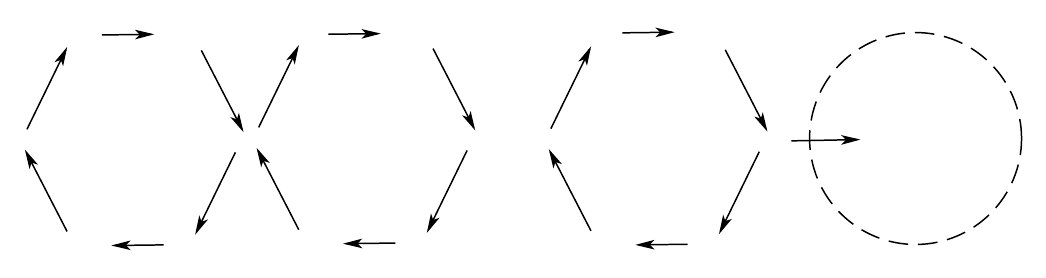}}
  \put(5.67,37.66){\fontsize{7.11}{8.54}\selectfont \textcolor[rgb]{0, 0, 0}{$\textbf{c}_0$}}
  \put(20.60,67.27){\fontsize{7.11}{8.54}\selectfont \textcolor[rgb]{0, 0, 0}{$\bullet$}}
  \put(50.47,8.25){\fontsize{7.11}{8.54}\selectfont \textcolor[rgb]{0, 0, 0}{$\bullet$}}
  \put(46.18,67.92){\fontsize{7.11}{8.54}\selectfont \textcolor[rgb]{0, 0, 0}{$\cdots$}}
  \put(19.14,7.20){\fontsize{7.11}{8.54}\selectfont \textcolor[rgb]{0, 0, 0}{$\cdots$}}
  \put(67.79,37.49){\fontsize{7.11}{8.54}\selectfont \textcolor[rgb]{0, 0, 0}{$\textbf{c}_1$}}
  \put(117.19,8.78){\fontsize{7.11}{8.54}\selectfont \textcolor[rgb]{0, 0, 0}{$\bullet$}}
  \put(85.87,7.73){\fontsize{7.11}{8.54}\selectfont \textcolor[rgb]{0, 0, 0}{$\cdots$}}
  \put(134.52,38.01){\fontsize{7.11}{8.54}\selectfont \textcolor[rgb]{0, 0, 0}{$\textbf{c}_2$}}
  \put(148.75,38.03){\fontsize{7.11}{8.54}\selectfont \makebox[0pt]{\textcolor[rgb]{0, 0, 0}{$\cdots$}}}
  \put(201.35,8.43){\fontsize{7.11}{8.54}\selectfont \textcolor[rgb]{0, 0, 0}{$\bullet$}}
  \put(170.02,7.38){\fontsize{7.11}{8.54}\selectfont \textcolor[rgb]{0, 0, 0}{$\cdots$}}
  \put(218.67,37.66){\fontsize{7.11}{8.54}\selectfont \textcolor[rgb]{0, 0, 0}{$\textbf{c}_s$}}
  \put(250.61,37.68){\fontsize{7.11}{8.54}\selectfont $\mathcal{G}'$}
  \put(85.83,67.48){\fontsize{7.11}{8.54}\selectfont \textcolor[rgb]{0, 0, 0}{$\bullet$}}
  \put(111.42,68.12){\fontsize{7.11}{8.54}\selectfont \textcolor[rgb]{0, 0, 0}{$\cdots$}}
  \put(170.49,67.89){\fontsize{7.11}{8.54}\selectfont \textcolor[rgb]{0, 0, 0}{$\bullet$}}
  \put(196.08,68.53){\fontsize{7.11}{8.54}\selectfont \textcolor[rgb]{0, 0, 0}{$\cdots$}}
  \end{picture}%
\caption{The firing graph of $P_n^k$ associated with $\c_0\in \B_s$}\label{firegraph}
\end{figure}

\begin{proof} (i) Let $\c_0\in \mathcal{B}_s$ be a fixed stable configuration on $P_n^k$, where $s\in [n]$.
Clearly,  $\tilde\c_0^i$ is a stable configuration on $K_k$ for each $i\in [n]$, where  $\tilde\c_0^i$ is defined at the beginning of this section.
By definition, $\c_0\in \mathcal{B}_s$ means that $\tilde\c_0^i\in \mathcal{C}_k$ for $1\le i\le s$ and $\tilde\c^{s+1}_0\in \mathcal{S}_k\backslash \mathcal{C}_k$.
Observe that to fire the vertices of $\e_i$ on the edge $\e_i$,
it is equivalent to consider the dollar game on a completed graph on vertices of $\e_i$ with the bank vertex $(i-1)(k-1)$.
Since $\tilde\c_0^1$ is a critical stable configuration, we get the subgraph of $\G(\c_0)$ by firing the vertices $0,1, \cdots, k-1$ on $\e_1$ and $k-1$ also $\e_2$; see Fig. \ref{1stcycle}.

\begin{figure}[H]
\centering
\ifpdf
  \setlength{\unitlength}{1bp}%
  \begin{picture}(171.85, 78.33)(0,0)
  \put(0,0){\includegraphics{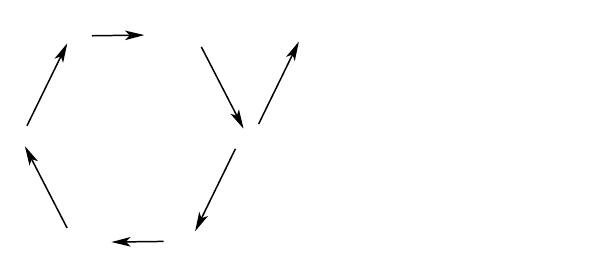}}
  \put(5.67,37.66){\fontsize{9.11}{8.54}{$\textbf{c}_0$}}
  \put(20.40,65.84){\fontsize{7.11}{8.54}\selectfont \textcolor[rgb]{0, 0, 0}{$\bullet$}}
  \put(50.47,8.25){\fontsize{7.11}{8.54}\selectfont \textcolor[rgb]{0, 0, 0}{$\bullet$}}
  \put(44.55,67.10){\fontsize{7.11}{8.54}\selectfont \textcolor[rgb]{0, 0, 0}{$\cdots$}}
  \put(19.14,7.20){\fontsize{7.11}{8.54}\selectfont \textcolor[rgb]{0, 0, 0}{$\cdots$}}
  \put(67.79,37.49){\fontsize{9.11}{8.54}\selectfont \textcolor[rgb]{0, 0, 0}{$\textbf{c}_1$}}
  \put(86.24,67.07){\fontsize{7.11}{8.54}\selectfont \textcolor[rgb]{0, 0, 0}{$\bullet$}}
  \put(82.90,50.59){\fontsize{7.11}{8.54}\selectfont $\alpha_{k-1, \textbf{e}_2}$}
  \put(66.34,21.55){\fontsize{7.11}{8.54}\selectfont $\alpha_{k-1, \textbf{e}_1}$}
  \end{picture}%
\caption{The directed cycle obtained by firing the vertices in $\e_1$} \label{1stcycle}
\end{figure}
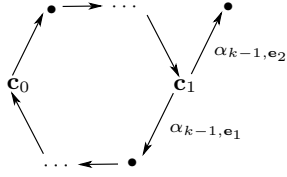

Similarly, since $\tilde\c_0^2, \cdots, \tilde\c_0^s$ are critical stable configurations,
each $\e_i$ yields a directed cycle of length $k$
by firing the vertices of $\e_i$ on $\e_i$, $i \in [s]$.
Hence, we get $s$ directed cycles, which implies that $\G'$ is not empty if and only if $s <n$.

(ii) Denote $$\V'=\{\c'\in \G(\c_0)\mid \c' \mbox{~is obtained from~} \c_s \mbox{~by first firing~} s(k-1) \mbox{~on~} \e_{s+1} \},$$
and $\W'(\c)=\sum\limits_{v=s(k-1)+1}^{n(k-1)} \c(v)$
for any $\c\in \G(\c_0)$.
We claim that
$\W'(\c')>\W'(\c)$ for any $\c'\in \V'$ and $\c\in \V(\G(\c_0))\backslash \V'$.


In fact, we choose a configuration $\c'_m \in \V'$ such that $\W'(\c'_m)=\min\{\W'(\c')\mid \c'\in \V'\}$, and
a directed path $P(\c_s, \c'_m)$ from $\c_s$ to $\c'_m$ as follows:
\[P(\c_s, \c'_m):
\xymatrix{\c_s\ar[rr]^{\al_{s(k-1),\e_{s+1}}} &&\c'_1 \ar[r]^{\al_{u'_1, \ast}} &\cdots \ar[r]& \c'_{m-1}\ar[r]^{\al_{u'_{m-1},\ast}}& \c'_m}.\]
We may also assume that $\W'(\c'_i)>\W'(\c'_m)$ for any $i \in [m-1]$.
By the firing rule, we have $\al_{u'_{m-1}, \ast}\neq \al_{u'_{m-1}, \e_j}$ for any $j \in [s]$; otherwise $\W'(\c'_{m-1})=\W'(\c'_m)$.
So there exists some configurations among $\c'_1,\ldots,\c'_{m-1}$ fired on some edge $\e_t$ for $t \ge s+1$.

If there exists $\al_{u'_i, \ast}=\al_{s(k-1), \e_{s+1}}$ for some $i \in [m-1]$, letting $u'_i$ be the last one for the vertex $s(k-1)$ fired on $\e_{s+1}$,
 then it suffices to consider the restriction of the configurations on $[s(k-1), n(k-1)]$,
which is equivalent to consider the firing graph of $P^k_{n-s}$ associated to the configuration $(\c'_i(s(k-1)), \cdots, \c'_i(n(k-1)))$ with $s(k-1)$ as the bank vertex,
where $P^k_{n-s}$ is the sub-hypergraph of $P_n^k$ consisting of the edges $\e_{s+1}, \cdots, \e_n$.
By Lemma \ref{t(k-1) fired iff} (ii), we have $\W'(\c'_m)\ge \W'(\c'_i)$, a contradiction.
So, form $\c'_1$ to $\c'_m$, $s(k-1)$ is no longer fired on $\e_{s+1}$.

By the above discussion, in order to compute the $\W'(\c'_m)$, it suffices to consider the firing graph of $P^k_{n-s}$ associated to the configuration
 $(\c_s(s(k-1)), \cdots, \c_s(n(k-1)))$ with $s(k-1)$ as the bank vertex.
Note that $\c_s(v)=\c_0(v)=\hat\c(v)$ for any $v >s(k-1)$ and any $\hat\c$ in the front $s$ directed cycles of $\G(\c_0)$.
By Lemma \ref{t(k-1) fired iff} (ii) again, we have $$\W'(\c'_m)>\W'(\c_s)=\W'(\c_0)=\W'(\hat\c),$$
since $\tilde\c_0^{s+1}=(\c_0(s(k-1)), \cdots, \c_0((s+1)(k-1)))$ is a non-critical.
So the claim follows, and the subgraph $\G'$ induced by $\V'$ contains no configurations in the front $s$ directed cycles of $\G(\c_0)$.

(iii)
If $s=0$, then Lemma \ref{t(k-1) fired iff}(ii), $\W(\c')>\W(\c_0)$ for any stable configuration $\c'$ in $\G'$ as $\tilde\c^1_0$ is not critical.
If $s\ge 1$, also by Lemma \ref{t(k-1) fired iff} (ii), we have $\W(\c')\ge \W(\c_0)$, and $\W(\c')=\W(\c_0)$
if and only if the vertices $1, \cdots, k-1$ have been fired on $\e_1$ exactly once along any directed path from $\c_0$ to $\c'$.
So it suffices to show there exists a vertex $u\in [n(k-1)]$ such that $\c_0(v)=\c'(v)$ for any $1\le v<u$ and $\c_0(u)>\c'(u)$ when $\W(\c')=\W(\c_0)$.

Suppose that $P(\c_0, \c')$ is a directed path from $\c_0$ to $\c'$ as follows:
\[P(\c_0, \c')\colon \xymatrix{
\c_0\ar[r]^{\al_{0,\e_1}} & \bar\c_0=\c_1\ar[r]^{\al_{u_1, \ast}} & \cdots\ar[r] & \c_s\ar[r]^{\al_{s(k-1), \e_{s+1}}} & \cdots\ar[r] & \c_t\ar[r]^{\al_{u_t, \ast}} & \c'}\]
As $\W(\c')=\W(\c_0)$, we have a unique $\sigma(v) \in [t]$ such that $\al_{u_{\sigma(v)}, \ast}=\al_{v, \e_1}$ for each $v \in [k-1]$ by Lemma \ref{t(k-1) fired iff} (ii).
It follows that $\c'(v)=\c_0(v)$ for $ v \in [k-2]$.

By construction, there exist $1\le i_1<j_1\le t$ such that
$\al_{u_{i_1}, \ast}=\al_{k-1, \e_2}$ and $\al_{u_{j_1}, \ast}=\al_{k-1, \e_1}$.
Note that the vertex $k-1$ is fired on $\e_1$ exactly once so that the arrow from $\c_{u_{j_1}}$ to $\c_{u_{j_1+1}}$ must be traveled.
If there exists $i_1<\ell_1<j_1$ satisfying $1\le u_{\ell_1}\le k-2$, then
by Lemma \ref{t(k-1) fired iff} (i),
\begin{align*}\c_{u_{i_1}}(k-1)&=\c_0(k-1)+1+m_1=k-1,\\
\c_{u_{j_1}}(k-1)&=\c_{u_{\ell_1}}(k-1)+m_2=k-1,
\end{align*}
 where $m_1$ (respectively, $m_2$) is the number of vertices of $[k-2]$ fired on $\e_1$ from $\c_1$ to $\c_{u_{i_1}-1}$ (respectively, from $\c_{u_{\ell_1}}$  to $\c_{u_{j_1}-1}$).
By the firing rule, $\c_{\ell_1}(k-1)<k-1$.
So,
 $$\c'(k-1)=k-1-(m_1+m_2+1)=\c_0(k-1)-((k-1)-\c_{u_{\ell_1}}(k-1))<\c_0(k-1),$$
 which implies that $\c_0\prec \c'$.
Otherwise, we have a unique $\sigma(v)\in [i_1+1,j_1-1]$ such that $\al_{u_{\sigma(v)},\ast}=\al_{v, \e_2}$ for each $v \in [k, 2(k-1)]$.
So $\c'(v)=\c_0(v)$ for $v \in [k-1,2(k-1)-1]$.
Then we continue to compare the values of $v$ given by $\c_0$ and $\c'$ for $v \ge 2k$.

Consider the firing graph of $P^k_{n-1}$ starting from the configuration
$(\c_{i_1}(k-1)$, $\cdots$, $\c_{i_1}(n(k-1)))=(\c_0(k), \cdots, \c_0(n(k-1))$,
where $P^k_{n-1}$ is the sub-hypergraph of $P_n^k$ consisting of the edge $\e_2, \cdots, \e_n$ with $k-1$ as the bank vertex.
Similar to the analysis for $P_n^k$, there exist $i_1<i_2<j_2<j_1$ such that
$\al_{u_{i_2},\ast}=\al_{2(k-1), \e_3}$ and $\al_{u_{j_2}, \ast}=\al_{2(k-1), \e_2}$.
If there exists $i_2<\ell_2<j_2$ satisfying $k\le u_{\ell_2}\le 2(k-1)-1$, then we have $\c'(2(k-1))<\c_0(2(k-1))$.
Otherwise, for each $v \in [2(k-1)+1,  3(k-1)]$, we have a unique $\sigma(v) \in [i_2+1,j_2-1]$ such that $\al_{u_{\sigma(v)}, \ast}=\al_{v, \e_3}$.
So $\c'(v)=\c_0(v)$ for $v \in [2(k-1),3(k-1)-1]$.
Then we continue to consider the firing graph of $P^k_{n-2}$ starting from the
configuration $(\c_{i_2}(2(k-1))$, $\cdots$, $\c_{i_2}(n(k-1)))=(\c_0(2(k-1)), \cdots, \c_0(n(k-1)))$, and so on.

By the above discussion, it follows that either $\c_0(v)=\c'(v)$ for $v \in [1, s'(k-1)-1]$ and $\c_0(s'(k-1))>\c'(s'(k-1))$ for some $s' \in [1,s-1]$,
or $\c_0(v)=\c'(v)$ for $v \in [1, s(k-1)-1]$.
If the former case occurs, then $\c_0 \prec \c'$.
Otherwise, there exist $i_1<i_2<\cdots < i_s<j_s<\cdots<j_2<j_1$ such that
$\al_{u_{i_s},\ast}=\al_{s(k-1), \e_{s+1}}$ and $\al_{u_{j_s}, \ast}=\al_{s(k-1), \e_s}$.
As $\tilde\c_0^{s+1}$ is a non-critical stable configuration,
there exists $i_s<\ell_s<j_s$ satisfying $(s-1)(k-1)+1\le u_{\ell_s} \le s(k-1)-1$, implying $\c'(s(k-1))<\c_0(s(k-1))$ by a similar discussion.
The result follows.
\end{proof}

\subsection{Formulas from firing graph}

Our goal is to compute the determinant of the left multiplication map by $f_w$ of $A=\dfrac{\C[x_v\colon v\in \widehat{V}]}{\lr{f_v\colon v\in \widehat{V}}}$ associated with
the hyperpath $P^k_n$ as in Fig. \ref{path}, where $w$ is taken to be the vertex $0$.
We focus on the firing graph $\G(\c_0)$ of the structure in Fig. \ref{firegraph}.
If a configuration $\c$ in $\G(\c_0)$ refers to a homogeneous polynomial $\x^\c=\prod\limits_{\c\in \widehat{V}} x_v^{\c(v)}$ by ignoring the bank vertex,
then we have
\[m_{f_0}(\x^{\c_0})=\la\x^{\c_0}-\x^{\widehat{\e}_1}\x^{\c_0}=\la\x^{\c_0}-\x^{\bar\c_0},\]
and
\[\x^\c=\sum_{(\c, \c')\in \E(\G(\c_0))} \dfrac{1}{\la}\x^{\c'}\]
for any $\c\neq \c_0$.

In view of this, we consider the weighted directed graph, still denoted by $\G(\c_0)$,
by assigning the weight $1$ on the arrow $\al_{0,\e_1}$ and the weight $\frac{1}{\la}$ on the others.
In order to obtain more explicit formulas, we need simplify the weighted firing graph $\G(\c_0)$.
Observe that the directed graph $\G(\c_0)$ may contain non-stable configurations.
Next, we will erase all non-stable configurations by modifying the weight until all of vertices are stable configurations.

Define a function $g^i(x)$ in indeterminant $\la$ recursively:
\begin{equation}\label{fung}
g^{-1}(x)=0, g^{0}(x)=1, g^{1}(x)=g(x)=\dfrac{1}{1-\frac{x}{\la^k}}, g^{i}(x)=g^{i-1}(g(x)) \mbox{~for~} i \ge 2.
\end{equation}

\begin{lem}\label{det of f_w in linear path}
Let $\c_0$ be a configuration in $\B_s$, $0\le s \le n$. Then we have
\[\x^{\bar\c_0}=\dfrac{g^{s-1}(1)}{\la^{k-1}}\x^{\c_0}+\sum_{\c'\in \SV(\G')} h_{\c'}(\la)\x^{\c'}.\]
where $\bar\c_0(v)=\begin{cases} \c_0(v)+1, & 1\le v\le k-1,\\ \c_0(v), &k\le v\le n(k-1)\end{cases}$,
$g^i(x)$ is defined as in (\ref{fung}), $h_{\c'}(\la)$ is a function in $\la$ for each $\c'$,
and $\SV(\G')$ is the set of all stable configurations in $\G'$ defined in Proposition \ref{Linear Path Firing Graph}.
\end{lem}

\begin{proof} Clearly, for $s=0$,
\[\x^{\bar\c_0}=\sum_{\c'\in \SV(\G')} h_{\c'}(\la)\x^{\c'}=\frac{g^{s-1}(1)}{\la^{k-1}}\x^{\c_0}+\sum_{\c'\in \S(\G')} h_{\c'}(\la)\x^{\c'},\]
and for $s=1$,
\[\x^{\bar\c_0}=\dfrac{1}{\la^{k-1}}\x^{\c_0}+\sum_{\c'\in \SV(\G')} h_{\c'}(\la)\x^{\c'}.\]
For $s\ge 2$, from the subgraph of $\G(\c_0)$ in Fig. \ref{s-cycle}, we get
\[\x^{\c_{s-1}}=\dfrac{1}{\la}\x^{\bar\c_{s-1}}+\dfrac{1}{\la^k}\x^{\c_{s-1}}+\dfrac{1}{\la^{l_s}}\x^{\c'_1},\]
where $l_s$ is the length of the directed path from $\c_{s-1}$ to $\c'_1$.
It follows that
\[\x^{\c_{s-1}}=g(1)\dfrac{1}{\la} \x^{\bar\c_{s-1}}+g(1)\dfrac{1}{\la^{l_s}}\x^{\c'_1}.\]
Therefore, we can erase the $s$-th directed cycle by adding an arrow $(\c_{s-1}, \c'_1)$ with weight $\dfrac{g(1)}{\la^{l_s}}$ and replacing the weight
of $(\c_{s-1}, \bar\c_{s-1})$ by $\dfrac{g(1)}{\la}$; see Fig. \ref{(s-1)-cycle}.

\begin{figure}[H]
\centering
\ifpdf
  \setlength{\unitlength}{1bp}%
  \begin{picture}(167.74, 96.21)(0,0)
  \put(0,0){\includegraphics{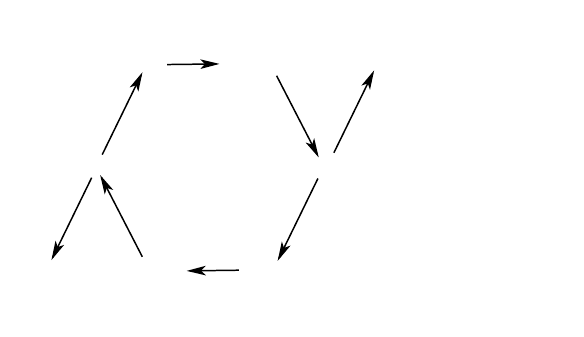}}
  \put(27.34,47.36){\fontsize{7.11}{8.54}\selectfont {$\textbf{c}_{s-1}$}}
  \put(42.06,75.55){\fontsize{7.11}{8.54}\selectfont {$\bullet$}}
  \put(72.14,17.96){\fontsize{7.11}{8.54}\selectfont {$\overline{\textbf{c}}_s$}}
  \put(66.22,76.80){\fontsize{7.11}{8.54}\selectfont {$\cdots$}}
  \put(40.81,16.90){\fontsize{7.11}{8.54}\selectfont {$\cdots$}}
  \put(89.46,47.19){\fontsize{7.11}{8.54}\selectfont {$\textbf{c}_s$}}
  \put(107.91,76.77){\fontsize{7.11}{8.54}\selectfont{$\textbf{c}'_1$}}
  \put(88.01,31.26){\fontsize{7.11}{8.54}\selectfont $\frac{1}{\lambda}$}
  \put(102.02,59.29){\fontsize{7.11}{8.54}\selectfont $\frac{1}{\lambda}$}
  \put(76.33,59.29){\fontsize{7.11}{8.54}\selectfont $\frac{1}{\lambda}$}
  \put(23.27,59.29){\fontsize{7.11}{8.54}\selectfont $\frac{1}{\lambda}$}
  \put(27.94,31.26){\fontsize{7.11}{8.54}\selectfont $\frac{1}{\lambda}$}
  \put(58.69,10.20){\fontsize{7.11}{8.54}\selectfont $\frac{1}{\lambda}$}
  \put(50.63,84.98){\fontsize{7.11}{8.54}\selectfont $\frac{1}{\lambda}$}
  \put(6.65,17.00){\fontsize{7.11}{8.54}\selectfont {$\overline{\textbf{c}}_{s-1}$}}
  \put(12.67,35.23){\fontsize{7.11}{8.54}\selectfont $\frac{1}{\lambda}$}
  \end{picture}%
\caption{The weighted $s$-th direct cycle by firing the vertices in $\e_s$}\label{s-cycle}
\end{figure}
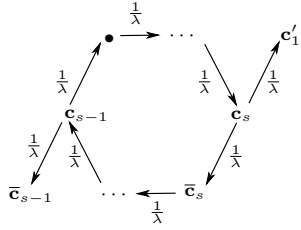

\begin{figure}[H]
\centering
\ifpdf
  \setlength{\unitlength}{1bp}%
  \begin{picture}(167.74, 96.21)(0,0)
  \put(0,0){\includegraphics{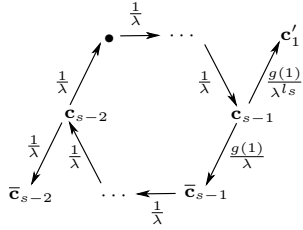}}
  \put(27.34,47.36){\fontsize{7.11}{8.54}\selectfont {$\textbf{c}_{s-2}$}}
  \put(42.06,75.55){\fontsize{7.11}{8.54}\selectfont {$\bullet$}}
  \put(72.14,17.96){\fontsize{7.11}{8.54}\selectfont {$\overline{\textbf{c}}_{s-1}$}}
  \put(66.22,76.80){\fontsize{7.11}{8.54}\selectfont {$\cdots$}}
  \put(40.81,16.90){\fontsize{7.11}{8.54}\selectfont {$\cdots$}}
  \put(89.46,47.19){\fontsize{7.11}{8.54}\selectfont {$\textbf{c}_{s-1}$}}
  \put(107.91,76.77){\fontsize{7.11}{8.54}\selectfont {$\textbf{c}'_1$}}
  \put(88.01,31.26){\fontsize{7.11}{8.54}\selectfont $\frac{g(1)}{\lambda}$}
  \put(102.02,59.29){\fontsize{7.11}{8.54}\selectfont $\frac{g(1)}{\lambda^{l_s}}$}
  \put(76.33,59.29){\fontsize{7.11}{8.54}\selectfont $\frac{1}{\lambda}$}
  \put(23.27,59.29){\fontsize{7.11}{8.54}\selectfont $\frac{1}{\lambda}$}
  \put(27.94,31.26){\fontsize{7.11}{8.54}\selectfont $\frac{1}{\lambda}$}
  \put(58.69,10.20){\fontsize{7.11}{8.54}\selectfont $\frac{1}{\lambda}$}
  \put(50.63,84.98){\fontsize{7.11}{8.54}\selectfont $\frac{1}{\lambda}$}
  \put(6.65,17.00){\fontsize{7.11}{8.54}\selectfont {$\overline{\textbf{c}}_{s-2}$}}
  \put(12.67,35.23){\fontsize{7.11}{8.54}\selectfont $\frac{1}{\lambda}$}
  \end{picture}%
\caption{The modified weighted $(s-1)$-th direct cycle by erasing the $s$-th directed cycle} \label{(s-1)-cycle}
\end{figure}

Let $l_j$ be the length of the directed path from $\c_{j-1}$ to $\c_j$ for $2\le j\le s-1$.
From this modified weighted subgraph, we immediately get
\[\x^{\c_{s-2}}=\dfrac{1}{\la}\x^{\bar\c_{s-2}}+g(1)\dfrac{1}{\la^k}\x^{\c_{s-2}}+g(1)\dfrac{1}{\la^{l_{s-1}+l_s}}\x^{\c'_1},\]
and
\[\x^{\c_{s-2}}=g^{2}(1)\dfrac{1}{\la} \x^{\bar\c_{s-2}}+g^{2}(1)g(1)\dfrac{1}{\la^{l_{s-1}+l_s}}\x^{\c'_1}.\]

Assume that $\x^{\c_{s-i}}=g^{i}(1)\dfrac{1}{\la} \x^{\bar\c_{s-i}}+\dfrac{g^{i}(1)\cdots g^{2}(1)g(1)}{\la^{l_{s-i+1}+l_{s-i+2}\cdots+l_s}}\x^{\c'_1}$ for $i\ge 1$.
Then  from the modified weighted subgraph (see Fig. \ref{(s-i)-cycle}),
 we have
\[\x^{\c_{s-(i+1)}}=
\dfrac{1}{\la} \x^{\bar\c_{s-(i+1)}}
+g^{i}(1)\dfrac{1}{\la^k} \x^{\c_{s-(i+1)}}
+\dfrac{g^{i}(1)\cdots g^{2}(1)g(1)}{\la^{l_{s-i}+l_{s-i+1}+\cdots+l_s}}\x^{\c'_1},\]
and hence
\[\x^{\c_{s-(i+1)}}=
g^{i+1}(1)\dfrac{1}{\la} \x^{\bar\c_{s-(i+1)}}
+\dfrac{g^{i+1}(1)\cdots g^{2}(1)g(1)}{\la^{l_{s-i}+l_{s-i+1}+\cdots+l_s}}\x^{\c'_1}.\]

\begin{figure}[H]
\centering
\ifpdf
  \setlength{\unitlength}{1bp}%
  \begin{picture}(167.74, 96.21)(0,0)
  \put(0,0){\includegraphics{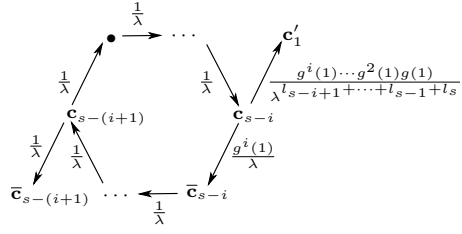}}
  \put(27.34,47.36){\fontsize{7.11}{8.54}\selectfont {$\textbf{c}_{s-(i+1)}$}}
  \put(42.06,75.55){\fontsize{7.11}{8.54}\selectfont {$\bullet$}}
  \put(72.14,17.96){\fontsize{7.11}{8.54}\selectfont {$\overline{\textbf{c}}_{s-i}$}}
  \put(66.22,76.80){\fontsize{7.11}{8.54}\selectfont {$\cdots$}}
  \put(40.81,16.90){\fontsize{7.11}{8.54}\selectfont {$\cdots$}}
  \put(89.46,47.19){\fontsize{7.11}{8.54}\selectfont {$\textbf{c}_{s-i}$}}
  \put(107.91,76.77){\fontsize{7.11}{8.54}\selectfont {$\textbf{c}'_1$}}
  \put(88.01,31.26){\fontsize{7.11}{8.54}\selectfont $\frac{g^{i}(1)}{\lambda}$}
  \put(102.02,59.29){\fontsize{7.11}{8.54}\selectfont $\frac{g^{i}(1)\cdots g^{2}(1)g(1)}{\lambda^{l_{s-i+1}+\cdots+l_{s-1}+l_s}}$}
  \put(76.33,59.29){\fontsize{7.11}{8.54}\selectfont $\frac{1}{\lambda}$}
  \put(23.27,59.29){\fontsize{7.11}{8.54}\selectfont $\frac{1}{\lambda}$}
  \put(27.94,31.26){\fontsize{7.11}{8.54}\selectfont $\frac{1}{\lambda}$}
  \put(58.69,10.20){\fontsize{7.11}{8.54}\selectfont $\frac{1}{\lambda}$}
  \put(50.63,84.98){\fontsize{7.11}{8.54}\selectfont $\frac{1}{\lambda}$}
  \put(6.65,17.00){\fontsize{7.11}{8.54}\selectfont {$\overline{\textbf{c}}_{s-(i+1)}$}}
  \put(12.67,35.23){\fontsize{7.11}{8.54}\selectfont $\frac{1}{\lambda}$}
  \end{picture}
 \caption{The  modified weighted $(s-i)$-th direct cycle by erasing the $(s-i+1)$-th, $\ldots$, $s$-th directed cycles}\label{(s-i)-cycle}
\end{figure}

Taking $i=s-2$, we get
\[\x^{\c_{1}}=
g^{s-1}(1)\dfrac{1}{\la} \x^{\bar\c_{1}}
+\dfrac{g^{s-1}(1)\cdots g^{2}(1)g(1)}{\la^{l_{2}+l_{3}+\cdots+l_s}}\x^{\c'_1}.\]
It follows that
\begin{align*}
\x^{\bar\c_0}=\dfrac{g^{s-1}(1)}{\la^{k-1}}\x^{\c_0}+\prod_{i=1}^{s-1} g^{i}(1) \dfrac{1}{\la^{\sum_{j=1}^{s}l_j}}\x^{\c'_1},
\end{align*}
where $l_1$ is the length of the directed path from $\bar\c_0$ to $\c_1$.
By Proposition \ref{Linear Path Firing Graph} (i), we know that there is not arrow from $\G'$ to $\G$ and therefore
$\x^{\c'_1}=\sum_{\c'\in \SV(\G')}\bar h_{\c'}(\la) \x^{\c'}$ for some function $\bar h_{\c'}(\la)$ in $\la$. It follows that
\[\x^{\bar\c_0}=\dfrac{g^{s-1}(1)}{\la^{k-1}}\x^{\c_0}+\sum_{\c'\in \SV(\G')} h_{\c'}(\la)\x^{\c'}.\]
\end{proof}

\subsection{The characteristic polynomial of hyperpaths}
We will give a recursive formula of the characteristic polynomial of hyperpaths.
Define
\begin{equation}\label{mu}
\mu_{n,k}(s)=\begin{cases}
k^{s(k-2)}((k-1)^{k-1}-k^{k-2})(k-1)^{(n-s-1)(k-1)}, & s \in [0, n-1],\\
k^{n(k-2)}, & s=n.
\end{cases}
\end{equation}

\begin{thm}\label{charpath}
Let $\phi^P_{n, k}(\la)$ be the characteristic polynomial of the hyperpath $P_n^k$, where $n\ge 2$. Then
\begin{align*}
\phi^P_{n, k}(\la)=\la^{(k-2)(k-1)^{n(k-1)}}\prod_{s=0}^{n}\left(\la-\dfrac{g^{s-1}(1)}{\la^{k-1}}\right)^{\mu_{n,k}(s)}\phi_{n-1, k}(\la)^{(k-1)^{k-1}}
\end{align*}
where $g^i(x)$ is defined in (\ref{fung}) and $\mu_{n,k}(s)$ is defined in (\ref{mu}).
\end{thm}

\begin{proof} By definition, the characteristic polynomial of $P_n^k$ is
\[\phi^P_{n, k}(\la)=\Res(F_0, F_1, \cdots, F_{r}),\]
where
\begin{align*}
F_i(x_0, x_1, \cdots, x_r)=\la x_i^{k-1}-\sum\limits_{i\in \e\in E} \x^{\e\backslash \{i\}}
\end{align*}
is a homogeneous polynomial in variables $x_0, x_1, \cdots, x_r$ of degree $k-1$ for $i=0, 1, \cdots, r=n(k-1)$.
Clearly,
\[
f_i(x_1, \cdots, x_r)=
\begin{cases}
\la-x_1\cdots x_{k-1}, & \hbox{~if~}i=0,\\
\la x_i^{k-1}-x_1\cdots x_{i-1}x_{i+1}\cdots x_{k-1}, & \hbox{~if~} i \in [k-2],\\
\la x_{k-1}^{k-1} -x_1\cdots x_{k-2}-x_k\cdots x_{2(k-1)}, & \hbox{~if~} i=k-1,\\
F_i(x_0, x_1, \cdots, x_r), & \hbox{~otherwise},
\end{cases}
\]
and
\[
\bar F_i(x_1, \cdots, x_r)=
\begin{cases}
\la x_i^{k-1}, & \hbox{~if~} i \in [k-2],\\
\la x_{k-1}^{k-1} -x_k\cdots x_{2(k-1)}, & \hbox{~if~}i=k-1,\\
F_i(x_0, x_1, \cdots, x_r), & \hbox{~otherwise}.
\end{cases}
\]
By Lemma \ref{Res(FG)=ResFResG} and \ref{Res(la)=laRes}, we have
\begin{align*}
\phi_{n, k}(\la)=&\Res(\bar F_1, \cdots, \bar F_r)^{k-1} \det(m_{f_0})\\
           =& \Res(\la x_1^{k-1}, \cdots, \la x_{k-2}^{k-1}, \bar F_{k-1}, \cdots, \bar F_r)^{k-1} \det (m_{f_0})\\
 =& \left(\Res(\la x_1^{k-1}, \cdots, \la x_{k-2}^{k-1})^{(k-1)^{(n-1)(k-1)+1}}\Res(\bar F_{k-1}, \cdots, \bar F_r)^{(k-1)^{k-2}}\right)^{k-1} \det (m_{f_0})\\
 =&\la^{(k-2)(k-1)^{n(k-1)}}\phi_{n-1, k}(\la)^{(k-1)^{k-1}}\det (m_{f_0})
\end{align*}
where $m_{f_0}$ is the multiplication map of the quotient algebra
\[A=\C[x_1, \cdots, x_r]/\langle f_1, \cdots, f_r\rangle\]
given by $m_{f_0}(x_1^{i_1}\cdots x_r^{i_r})=\la x_1^{i_1}\cdots x_r^{i_r}-x_1^{i_1+1}\cdots x_{k-1}^{i_{k-1}+1}x_{k}^{i_k} \cdots x_{r}^{i_r}$.

We choose a $\C$-basis $\mathbf{B}=\{\x^\c\mid \c\colon [r]\to [k-2]\}$ for $A$, where $\x^\c=x_1^{\c(1)}\cdots x_r^{\c(r)}$.
In fact, for any $\x^\c \in \mathbf{B}$, $\tilde\c=(\tilde\c(0),\c(1),\ldots,\c(r))$ can be viewed as a stable configuration on $P_n^k$, where
$0$ is the bank vertex whose value can be omitted.
We denote by $\B$ be the set of all stable configurations on $P_n^k$. Clearly, there is a one-to-one correspondence between $\mathbf{B}$ and $\B$ by ignoring the bank vertex,
and the left anti-lexicographical ordering on $\B$ gives a total ordering on $\mathbf{B}$.
To be precise, $x_1^{i_1}\cdots x_r^{i_r}\prec x_1^{j_1}\cdots x_r^{j_r}$
if and only if $\Sum_{t=1}^r i_t<\Sum_{t=1}^r j_t$, or $\Sum_{t=1}^r i_t=\Sum_{t=1}^r j_t$,
$i_s=j_s$ for $1\le s\le t-1$ and $i_t>j_t$ for some $t\in [r]$.

Retaining the notation in Section 4.2, we observe that $\B$ is exactly the disjoint union of $\B_0, \B_1, \cdots, \B_n$.
For each $s \in [n]$, the number of configurations in $\B_s$ is exactly $\mu_{n,k}(s)$.
Let $\c_0$ be a configuration in $\B_s$. By Lemma \ref{det of f_w in linear path}, we know
\[m_{f_0}(\x^{\c_0})=\x^{\bar\c_0}=\left(\la-\dfrac{g^{s-1}(1)}{\la^{k-1}}\right)\x^{\c_0}-\sum_{\c'\in \SV(\G')} h_{\c'}(\la)\x^{\c'},\]
where $\G'$, $\mathcal{S}(\G')$ and $h_{\c'}(\la)$ are defined in Proposition \ref{Linear Path Firing Graph} or Lemma \ref{det of f_w in linear path}.
On the other hand, by Proposition \ref{Linear Path Firing Graph} (ii), we have $\c_0\prec \c'$ for any $\c'$ in $\G'$.
It follows that the matrix of $m_{f_0}$ associated to the ordered basis $\mathbf{B}$ is a lower triangle matrix
with $\frac{g^{s-1}(1)}{\la^{k-1}}$ appearing on the diagonal exactly  $\mu_{n,k}(s)$ times for $s \in [0,n]$.
So
\[\det(m_{f_0})=\prod_{s=0}^{n}\left(\la-\dfrac{g^{s-1}(1)}{\la^{k-1}}\right)^{\mu_{n,k}(s)},\]
and
\begin{align*}
\phi^P_{n, k}(\la)=
\la^{(k-2)(k-1)^{n(k-1)}}\prod_{s=0}^{n}\left(\la-\dfrac{g^{s-1}(1)}{\la^{k-1}}\right)^{\mu_{n,k}(s)}
\phi_{n-1, k}(\la)^{(k-1)^{k-1}}.
\end{align*}
\end{proof}

\begin{exm} By Theorem \ref{charpath}, we get the characteristic polynomial $\phi^P_{n, k}(\la)$ of $P_n^k$ for some specified $n$ and $k$.

$\phi^P_{1,3}(\la)=\la^3(\la^3-1)^3$ of degree $12$;

$\phi^P_{2, 3}(\la)=\la^{35}(\la^3-1)^6(\la^3-2)^9$ of degree $80$;

$\phi^P_{3, 3}(\la)=\la^{151}(\la^3-1)^{27}(\la^3-2)^{18}(\la^6-3\la^3+1)^{27}$ of degree $448$;

$\phi^P_{4, 3}(\la)=\la^{891}(\la^3-1)^{201}(\la^3-2)^{81}(\la^3-3)^{81}(\la^6-3\la^3+1)^{54}$ of degree $2304$;

$\phi^P_{1, 4}(\la)=\la^{44}(t^4-1)^{16}$ of degree $108$;

$\phi^P_{2, 4}(\la)=\la^{2671}(\la^4-1)^{352}(\la^4-2)^{256}$ of degree $5103$;

$\phi^P_{3, 4}(\la)=\la^{95774}(\la^4-1)^{11440}(\la^4-2)^{5632}(\la^8-3\la^4+1)^{4096}$ of degree $196830$.
\end{exm}

\section{Starlike hypergraphs}
In the last section we will deal with a class of $k$-uniform hypergraphs, called \emph{starlike hypergraph} and denoted by $S^k_{n_1,\ldots,n_m}$, which is obtained from $m$ hyperpaths of length $n_1,\ldots,n_m$ by sharing a common vertex $w$, where $m \ge 1$ and $n_i \ge 1$ for $i \in [m]$; see Fig. \ref{starlike}.
When $m=1$, it is a hyperpath $P_{n_1}^k$.
When $n_i=1$ for each $i \in [m]$, it is called a \emph{hyperstar} with $m$ edges and denoted by $S^k_m$.

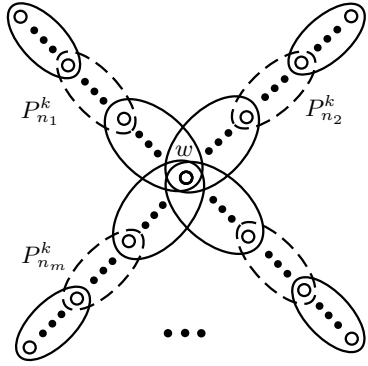
\begin{figure}[htbp]
\centering
\ifpdf
  \setlength{\unitlength}{1bp}%
   \begin{picture}(168.82, 145.40)(0,0)
  \put(0,0){\includegraphics{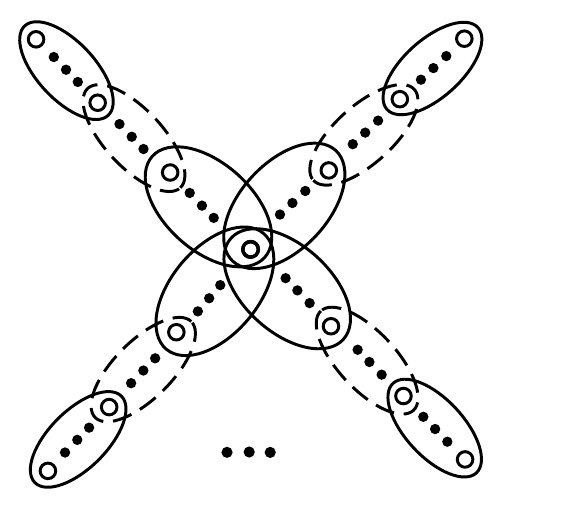}}
  \put(68.01,82.28){\fontsize{8.54}{10.24}\selectfont $w$}
  \put(10.50,96.64){\fontsize{8.54}{10.24}\selectfont $P^k_{n_1}$}
  \put(116.51,97.28){\fontsize{8.54}{10.24}\selectfont $P^k_{n_2}$}
  \put(10.67,42.10){\fontsize{8.54}{10.24}\selectfont $P^k_{n_m}$}
  \end{picture}%
  \caption{The starlike hypergraph $S^k_{n_1,n_2,\ldots,n_m}$}\label{starlike}
\end{figure}

\begin{thm}\label{shchar}
Let $S^k_{n_1,\ldots,n_m}$ be the $k$-uniform starlike hypergraph with exactly $t$ hyperpaths having length $1$, where $0 \le t \le m$.
Then the characteristic polynomial $\phi_{n_1,\ldots,n_m;k}^{SH}(\la)$ of $S^k_{n_1,\ldots,n_m}$ is
$$ \la^{(m(k-2)+t)(k-1)^{\Sum_{j=1}^m r_j}} \Prod_{i \in [m]\atop n_i>1} \phi^P_{n_i-1,k}(\la)^{(k-1)^{k-1+\Sum_{j \ne i}r_j}}
\Prod_{s_i \in [0,n_i]\atop 1 \le i \le m}\left(\la-\Sum_{i=1}^{m}\dfrac{g^{s_i-1}(1)}{\la^{k-1}}\right)^{\Prod_{i=1}^m \mu_{n_i,k}(s_i)},$$
where $r_i=n_i(k-1)$ for $i \in [m]$, $\phi^P_{n,k}(\la)$ denotes the characteristic polynomial of $P_n^k$,
$g^s(x)$ and $\mu_{n,k}(s)$ are defined in (\ref{fung}) and (\ref{mu}) respectively.
\end{thm}

\begin{proof}
For each $i \in [m]$, let $\e_i$ be the edge of $P_{n_i}^k$ containing $w$, and $\widehat\e_i=\e_i \backslash \{w\}$.
Let $H_i$ be sub-hyperpath of $P_{n_i}^k$ by deleting the vertex $w$ and the edge $\e_i$ for $i \in [m]$,
which is a hyperpath $P_{n_i-1}^k$ together with $k-2$ isolated vertices if $n_i>1$ or $r_i(=k-1)$ isolated vertices otherwise.
If $n_i>1$, by Lemma \ref{Res(FG)=ResFResG} and Lemma \ref{Res(la)=laRes}, we get
\begin{equation}\label{SH1}\phi_{H_i}(\la)=\la^{(k-2)(k-1)^{r_i-1}}\phi_{n_i-1,k}^P(\la)^{(k-1)^{k-2}};\end{equation}
otherwise,
\begin{equation}\label{SH2}\phi_{H_i}(\la)=\la^{(k-1)(k-1)^{r_i-1}}.\end{equation}
So, combining (\ref{SH1}) and (\ref{SH2}), we have
\begin{equation}\label{SH3}
\Prod_{i=1}^m \phi_{H_i}^{(k-1)^{1+\Sum_{j \ne i} r_j}}=\la^{(m(k-2)+t)(k-1)^{\Sum_{j=1}^m r_j}} \Prod_{i \in [m]\atop n_i>1} \phi^P_{n_i-1,k}(\la)^{(k-1)^{k-1+\Sum_{j \ne i}r_j}}.
\end{equation}

Let $V$ be the vertex set of $S^k_{n_1,\ldots,n_m}$, and let $V_i$ be the vertex set of $P_{n_i}^k$ for $i \in [m]$.
Note that $w$ is a cut vertex of $S^k_{n_1,\ldots,n_m}$ and recall the discussion in Section 3.2.
The quotient algebra $A=\dfrac{\C[x_v: v \in V\backslash \{w\}]}{\lr{f_v:v \in V\backslash \{w\}}}$
holds
$$ A =A_1 \otimes \cdots \otimes A_m,$$
where $A_i=\dfrac{\C[x_v: v \in V_i\backslash \{w\}]}{\lr{f_v:v \in V_i\backslash \{w\}}}$ for $i \in [m]$,
and
$$f_v=\la x_v^{k-1}-\Sum_{v \in \e \in E(P_{n_i}^k)}x_{\e \backslash \{v\}}-\Sum_{v \in \widehat\e_i}x_{\widehat\e_i \backslash \{v\}}, \; v \in V_i\backslash \{w\}, i \in [m].$$
Let
$$f_w=\la-\Sum_{i=1}^m x_{\widehat\e_i}$$
and $m_{f_w} \colon A \to A$ be the linear map given by multiplication by $f_w$.
Then
$$m_{f_w}=\la \Id_A-\Sum_{i=1}^m \Id_{A_1} \otimes \cdots \otimes \Id_{A_{i-1}} \otimes m_{i,w} \otimes \Id_{A_{i+1}} \otimes \cdots \otimes \Id_{A_m},$$
where $m_{i,w}: A_i \to A_i$ is the linear map given by multiplication by $\x_{\widehat\e_i}$ for $i \in [m]$.

Fix an $i \in [m]$.
Let $\mathbf{B}=\{\x^{\c}|\c: V_i\backslash \{w\} \to [k-2]\}$ be a basis of $A_i$, where $\x^{\c}=\prod_{v \in V_i\backslash \{w\}}x_v^{\c(v)}$,
and let $\mathcal{B}$ be the set of stable configuration of $P^k_{n_i}$ with $w$ as the bank vertex.
As discussed in Theorem \ref{charpath}, there is a one-to-one correspondence between $\mathbf{B}$ and $\mathcal{B}$ by ignoring the bank vertex.
We also have a left anti-lexicographical ordering $\prec$ on $\mathbf{B}$ arising from the order of $\mathcal{B}$.
Retaining the notation in Section 4.2, $\mathcal{B}$ is the disjoint union of $\mathcal{B}_0,\mathcal{B}_1,\ldots,\mathcal{B}_{n_i}$,
and the number of configurations in $\mathcal{B}_s$ is $\mu_{n_i,k}(s)$ for each $s \in [0,n_i]$.

For each $\c_0 \in \mathcal{B}_s$, by Lemma \ref{det of f_w in linear path},
$$m_{i,w}(\x^{\c_0})=\x^{\bar\c_0}=\dfrac{g^{s-1}(1)}{\la^{k-1}}\x^{\c_0}+\sum_{\c' \in \mathcal{S}(\G')}h_{\c'}(\la)\x^{\c'},$$
where $\G'$, $\mathcal{S}(\G')$ and $h_{\c'}(\la)$ are defined in Proposition \ref{Linear Path Firing Graph} or Lemma \ref{det of f_w in linear path}.
Note that $\c_0 \prec \c'$ for any $c' \in \G'$.
So the matrix of $m_{i,w}$ associated with the basis $\mathbf{B}$ under the above order  is a lower triangle matrix with
$\frac{g^{s-1}(1)}{\la^{k-1}}$ appearing on the diagonal exactly  $\mu_{n_i,k}(s)$ times for $s \in [0,n_i]$.

By the above discussion,
\begin{equation}\label{SH4}
\det(m_{f_w})=\Prod_{s_i \in [0,n_i]\atop 1 \le i \le m}\left(\la-\Sum_{i=1}^{m}\dfrac{g^{s_i-1}(1)}{\la^{k-1}}\right)^{\Prod_{i=1}^m \mu_{n_i,k}(s_i)}.
\end{equation}
The result follows by Corollary \ref{charcut} and the equalities (\ref{SH3}) and (\ref{SH4}).
\end{proof}

Taking $n_i=1$ for $i \in [m]$ in Theorem \ref{shchar}, we get the characteristic polynomial of hyperstar $S_m^k$.

\begin{cor}
Let $\phi_{n,k}^S$ be the characteristic polynomial of the $k$-uniform hyperstar $S_m^k$ with $m$ edges.
Then
\begin{align*}
\phi_{m,k}^S(\la)=\la^{r(k-1)^r} \prod\limits_{p=0}^m \left(\la-\dfrac{p}{\la^{k-1}}\right)^{{m \choose p} k^{(k-2)p}\left((k-1)^{{k-1}}-k^{k-2}\right)^{n-p} }
\end{align*}
where $r=m(k-1)$.
\end{cor}

\begin{cor}\cite[Theorem 4.3]{CD} Let $E$ be the $k$-uniform hypergraph with $k$ vertices and a single edge. Then
\[\phi_E(\la)=\la^{k(k-1)^{k-1}-k^{k-1}}(\la^k-1)^{k^{k-2}}.\]
\end{cor}

\begin{cor}Let $S_2^k$ be a $k$-uniform hyperstar with two edges or hyperpath with two edges. Then
\[\phi_{S_2^k}(\la)=\la^{\mu_k}(\la^k-1)^{2k^{k-2}\left((k-1)^{k-1}-k^{k-2}\right)}(\la^k-2)^{k^{2(k-2)}},\]
where $\mu_k=(2k-1)(k-1)^{2(k-1)}-2k^{k-1}(k-1)^{k-1}+k^{2k-3}$.
\end{cor}

\begin{cor}
The characteristic polynomial of the starlike hypergraph $S^k_{1,1,2}$ is
\begin{align*}
\phi_{1,1,2;k}^{SH}(\la)&=\la^{(4k-3)(k-1)^{4(k-1)}-(k^{k-1}+3k^{k-2})(k-1)^{3(k-1)}+3k^{2(k-2)}(k-1)^{2(k-1)}-k^{3(k-2)}(k-1)^{k-1}}\\
&\cdot (\la^k-1)^{k^{k-2}(k-1)^{3(k-1)}}\\
&\cdot \left(\la-\frac{1}{\la^{k-1}}\right)^{k^{k-2}(3(k-1)^{k-1}-k^{k-2})((k-1)^{k-1}-k^{k-2})^2}\\
&\cdot \left(\la-\frac{2}{\la^{k-1}}\right)^{k^{2(k-2)}(3(k-1)^{k-1}-2k^{k-2})((k-1)^{k-1}-k^{k-2})}\\
&\cdot \left(\la-\frac{3}{\la^{k-1}}\right)^{k^{3(k-2)}((k-1)^{k-1}-k^{k-2})}\\
&\cdot \left(\la-\frac{\la}{\la^k-1}\right)^{k^{2(k-2)}((k-1)^{k-1}-k^{k-2})^2}\\
&\cdot \left(\la-\frac{1}{\la^{k-1}}-\frac{\la}{\la^k-1}\right)^{2k^{3(k-2)}((k-1)^{k-1}-k^{k-2})}\\
&\cdot \left(\la-\frac{2}{\la^{k-1}}-\frac{\la}{\la^k-1}\right)^{k^{4(k-2)}}.
\end{align*}
\end{cor}

In particular,
the characteristic polynomial $\phi_{1,1,2;3}^{SH}(\la)$ of $S^3_{1,1,2}$ is
$$\la^{980}(\la^3-1)^{75}(\la^3-2)^{54}(\la^3-3)^{27}(\la^4-2\la)^9(\la^6-3\la^3+1)^{54}(\la^6-4\la^3+2)^{81}$$
of degree $2294$.

\section{Conculusion}
We give an explicit and recursive formula for the characteristic polynomial of the adjacency tensor of a starlike hypergraph,
  which is a resultant of a system of polynomials related to the structure of the hypergraph.
The variants of chip-firing game on simple graphs such as dollar game on simple graphs or hypergraphs are applied to the such kinds of resultants.
So we provide a combinatorial method for computing resultants, which will have potential use for commutative algebra, algebraic geometry and physical fields.

We note that there are many numerical methods and algorithms for computing partial (real or extreme) eigenvalues of a general (symmetric) tensor;
see e.g. Chang et.al \cite{Chang} and the references therein.
We also note the starlike hypergraph is a power hypergraph $G^k$, which is obtained from a starlike simple graph $G$ by adding $k-2$ vertices to each of its edges.
Zhou et.al \cite{Zhou} proved that if $\la$ is a nonzero eigenvalue of $G$ or any subgraph of $G$, then $\la^{\frac{2}{k}}$ is an eigenvalue of $G^k$.
In fact, the nonzero eigenvalues of  $G^k$ are exactly those eigenvalues arising from $G$ in the above way.
However we do not know the algebraic multiplicities of the eigenvalues (including the zero eigenvalue) from their result.

\subsection*{Acknowledgments}
This work is supported by supported by Natural Science Foundation of China (Grant No. 11771016 and 61501003)

\vspace{0.5cm}

{\footnotesize
\noindent Yan-Hong Bao \\
School of Mathematical Sciences, Anhui University, Hefei 230601, China\\
E-mail address: baoyh@ahu.edu.cn

\vspace{2mm}
\noindent Yi-Zheng Fan \\
School of Mathematical Sciences, Anhui University, Hefei 230601, China\\
E-mail address: fanyz@ahu.edu.cn

\vspace{2mm}
\noindent Yi Wang \\
School of Mathematical Sciences, Anhui University, Hefei 230601, China\\
E-mail address: wangy@ahu.edu.cn

\vspace{2mm}
\noindent Ming Zhu \\
School of Electronics and Information Engineering, Anhui University, Hefei 230601, China\\
E-mail address: zhuming@ahu.edu.cn
}

\end{document}